\documentclass[11pt,a4paper,twoside]{amsart}
\usepackage[top=1in, bottom=1.25in, marginparwidth=1in, inner=1in, outer=1in]{geometry}

\usepackage{lmodern}
\usepackage[T2A,T1]{fontenc} 
\usepackage[UKenglish]{babel}
\usepackage{bm}

\makeatletter
\newcommand*{\math@version@bold}{bold}
\DeclareMathOperator\DD{
	\textrm{%
		\usefont{T2A}{cmr}{\ifx\math@version\math@version@bold bx\else m\fi}{n}%
		\CYRD
	}%
} 
\makeatother

\usepackage{graphicx}
\usepackage{xcolor}
\usepackage{tikz-cd}
\usepackage{pinlabel}

\usepackage{caption}
\usepackage[labelfont=up,margin=5pt]{subcaption}

\usepackage[hidelinks,bookmarks=true,pdfencoding=auto,bookmarksdepth=3]{hyperref}
\usepackage{setspace} 
\usepackage{amsmath}
\usepackage{amsfonts}
\usepackage{amssymb}
\usepackage{amsthm}
\usepackage{mathtools}
\mathtoolsset{mathic=true}
\usepackage{nicefrac}

\newtheorem{theorem}{Theorem}[section]
\newtheorem*{theorem*}{Theorem}
\newtheorem{lemma}[theorem]{Lemma}

\newtheorem*{csconj}{Cosmetic Surgery Conjecture}
\newtheorem*{ccconj}{Cosmetic Crossing Conjecture}
\newtheorem{conjecture}[theorem]{Conjecture}

\theoremstyle{definition}
\newtheorem{definition}[theorem]{Definition}
\newtheorem{example}[theorem]{Example}

\newtheorem{remark}[theorem]{Remark}

\usepackage{enumitem}

\makeatletter
\newcommand{\myitem}[1]{%
	\item[#1]\protected@edef\@currentlabel{#1}%
}
\makeatother

\newcommand{\QGrad}[1]{{\textcolor{violet}{#1}}}

\newcommand{\HomGrad}[1]{{\textcolor{black}{#1}}}

\newcommand{\GGzqh}[4]{\prescript{\QGrad{#3}}{}{#1}_\HomGrad{#4}}
\newcommand{\GGzqz}[4]{\prescript{\QGrad{#3}}{}{#1}}

\DeclareMathOperator{\HFT}{HFT}

\renewcommand{\a}{\mathbf{a}}
\newcommand{\eight}{\mathbf{e}}
\renewcommand{\c}{c}
\newcommand{\tc}{\tilde{\c}}
\newcommand{\s}{\mathbf{s}}
\newcommand{\ts}{\tilde{\s}}
\renewcommand{\r}{\mathbf{r}}
\newcommand{\tr}{\tilde{\r}}
\newcommand{\rKh}{\mathbf{r}}
\newcommand{\sKh}{\mathbf{s}}

\newcommand{\KhTl}[1]{[\![ #1 ]\!]_{/l}} 

\DeclareMathOperator{\Kh}{Kh}
\DeclareMathOperator{\AKh}{AKh}
\DeclareMathOperator{\BN}{BN}
\newcommand{\Khr}{\widetilde{\Kh}}
\newcommand{\BNr}{\widetilde{\BN}}
\newcommand{\HF}{\operatorname{HF}}
\DeclareMathOperator{\Homology}{\mathbf{H}_\ast}

\newcommand{\mirror}{\operatorname{m}} 

\newcommand{\PuncturedPlane}{\mathbb{R}^2\smallsetminus \mathbb{Z}^2}

\newcommand{\QPI}{\operatorname{\mathbb{Q}P}^1}

\newcommand{\FourPuncturedSphereKh}{S^2_{4,\ast}}
\newcommand{\pt}{{\text{pt}}}

\DeclareMathOperator{\BNAlgH}{\mathcal{B}}

\DeclareMathOperator{\id}{id}

\DeclareMathOperator{\PSL}{PSL}

\def\F{\mathbb{F}}

\def\R{\mathbb{R}}
\def\Z{\mathbb{Z}}

\renewcommand{\C}{\mathbb{C}}
\newcommand{\fieldTwoElements}{\mathbb{F}}

\DeclareMathOperator{\Mor}{Mor}
\DeclareMathOperator{\Mod}{Mod}

\DeclareMathOperator{\Cob}{Cob}

\newcommand{\hateqq}{\mathrel{\widehat{=}}}
\def\co{\colon\thinspace\relax}

\newcommand{\vc}[1]{\vcenter{\hbox{#1}}}%
\newcommand{\mypic}[3]{%
	\newcommand{#3}{%
		\vc{%
			\includegraphics[page=#2]%
			{PSTricks/PSTricks-figures-#1-pics.pdf}%
		}%
	}%
}%

\mypic{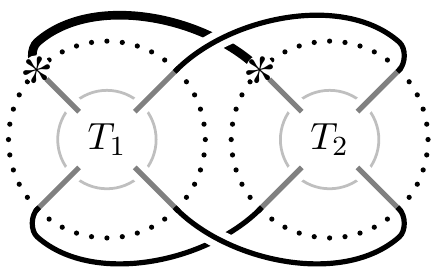}{1}{\tanglepairingI}
\mypic{main}{2}{\tanglepairingII}
\mypic{main}{3}{\inlineTrivialTangle}

\mypic{main}{4}{\Tpii}
\mypic{main}{5}{\Tmii}
\mypic{main}{6}{\Tpn}
\mypic{main}{7}{\Tmn}
\mypic{main}{8}{\Tpiin}
\mypic{main}{9}{\Tmiin}
\mypic{main}{10}{\splittangle}
\mypic{main}{11}{\ECSCIspecials}
\mypic{main}{12}{\ECSCIIspecials}
\mypic{main}{13}{\AGCCCzeroSpecials}
\mypic{main}{14}{\AGCCCzeroRationals}
\mypic{main}{15}{\pretzeltangleKh}
\mypic{main}{16}{\pretzeltangleDownstairsKhr}
\mypic{main}{17}{\pretzeltangleUpstairsKhr}
\mypic{main}{18}{\pretzeltangleDownstairsBNr}
\mypic{main}{19}{\pretzeltangleUpstairsBNr}
\mypic{main}{20}{\PairingTrefoilArcINTRO}
\mypic{main}{21}{\PairingTrefoilLoopINTRO}
\mypic{main}{22}{\PairingBothINTRO}
\mypic{main}{23}{\GeographyCovering}
\mypic{main}{24}{\Blanksphere}
\mypic{main}{25}{\slopeZeroNbh}
\mypic{main}{27}{\KissingHearts}

\mypic{main}{28}{\tanglepairing}
\mypic{main}{29}{\ratixo}
\mypic{main}{30}{\ratoxi}
\mypic{main}{31}{\ratixi}
\mypic{main}{32}{\ratixii}
\mypic{main}{33}{\ratixiii}
\mypic{main}{34}{\crossingcircle}
\mypic{main}{35}{\AGCCCnoupperbound}

\mypic{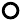}{1}{\DotC}
\mypic{inline}{2}{\DotB}
\mypic{inline}{3}{\DotCblue}
\mypic{inline}{4}{\DotBblue}
\mypic{inline}{5}{\DotCred}
\mypic{inline}{6}{\DotBred}
\mypic{inline}{1,scale=0.7}{\smallDotC}
\mypic{inline}{2,scale=0.7}{\smallDotB}
\mypic{inline}{3,scale=0.7}{\smallDotCblue}
\mypic{inline}{4,scale=0.7}{\smallDotBblue}
\mypic{inline}{5,scale=0.7}{\smallDotCred}
\mypic{inline}{6,scale=0.7}{\smallDotBred}

\mypic{inline}{7}{\Circle}
\mypic{inline}{8}{\resolution}

\mypic{inline}{9}{\Li}
\mypic{inline}{10}{\Lo}
\mypic{inline}{10,angle=270,origin=c}{\Lialt}
\mypic{inline}{11}{\LiRed}
\mypic{inline}{12}{\LoRed}
\mypic{inline}{13}{\LiBlue}
\mypic{inline}{14}{\LoBlue}
\mypic{inline}{15}{\Ni}
\mypic{inline}{16}{\No}
\mypic{inline}{15,scale=0.7}{\smallNi}
\mypic{inline}{16,scale=0.7}{\smallNo}
\mypic{inline}{17}{\Lil}
\mypic{inline}{18}{\Lol}
\mypic{inline}{19}{\Nil}
\mypic{inline}{20}{\Nol}
\mypic{inline}{21}{\CrossingL}
\mypic{inline}{22}{\CrossingR}
\mypic{inline}{23}{\CrossingLDot}
\mypic{inline}{24}{\CrossingRDot}
\mypic{inline}{25}{\CrossingPos}
\mypic{inline}{26}{\CrossingNeg}
\mypic{inline}{27}{\CrossingPosMarkedi}
\mypic{inline}{28}{\CrossingNegMarkedi}
\mypic{inline}{29}{\CrossingPosMarkedii}
\mypic{inline}{30}{\CrossingNegMarkedii}
\mypic{inline}{31}{\CrossingPosMarkediii}
\mypic{inline}{32}{\CrossingNegMarkediii}
\mypic{inline}{33}{\CrossingPosMarkediv}
\mypic{inline}{34}{\CrossingNegMarkediv}
\mypic{inline}{25,angle=270,origin=c}{\CrossingPosR}
\mypic{inline}{26,angle=270,origin=c}{\CrossingNegR}
\mypic{inline}{27,angle=270,origin=c}{\CrossingPosMarkediR}
\mypic{inline}{28,angle=270,origin=c}{\CrossingNegMarkediR}
\mypic{inline}{29,angle=270,origin=c}{\CrossingPosMarkediiR}
\mypic{inline}{30,angle=270,origin=c}{\CrossingNegMarkediiR}
\mypic{inline}{31,angle=270,origin=c}{\CrossingPosMarkediiiR}
\mypic{inline}{32,angle=270,origin=c}{\CrossingNegMarkediiiR}
\mypic{inline}{33,angle=270,origin=c}{\CrossingPosMarkedivR}
\mypic{inline}{34,angle=270,origin=c}{\CrossingNegMarkedivR}

\mypic{inline}{35}{\LiO}
\mypic{inline}{36}{\LiOl}
\mypic{inline}{37}{\LilO}
\mypic{inline}{38}{\LiOr}
\mypic{inline}{39}{\LirO}
\mypic{inline}{40}{\NiO}
\mypic{inline}{41}{\NiOl}
\mypic{inline}{42}{\NilO}
\mypic{inline}{43}{\NiOr}
\mypic{inline}{44}{\NirO}
\mypic{inline}{45}{\LiDotL}
\mypic{inline}{46}{\LiDotR}
\mypic{inline}{47}{\LoDotT}
\mypic{inline}{48}{\LoDotB}
\mypic{inline}{49}{\NiDotL}
\mypic{inline}{50}{\NiDotR}
\mypic{inline}{51}{\NoDotT}
\mypic{inline}{52}{\NoDotB}
\mypic{inline}{53}{\TwistTwo}
\mypic{inline}{53,angle=90,origin=c}{\TwistTwoUp}
\mypic{inline}{54}{\TwistNegTwo}
\mypic{inline}{55}{\TwistThree}
\mypic{inline}{56}{\TwistNegThree}
\mypic{inline}{57}{\TwistTwoDot}
\mypic{inline}{58}{\TwistNegTwoDot}
\mypic{inline}{59}{\TwistThreeDot}
\mypic{inline}{60}{\TwistNegThreeDot}

\mypic{inline}{61}{\CrossingLBasepoint}
\mypic{inline}{61,angle=180,origin=c}{\CrossingLBasepointRotated}
\mypic{inline}{62}{\CrossingRBasepoint}
\mypic{inline}{62,angle=180,origin=c}{\CrossingRBasepointRotated}

\mypic{inline}{63}{\LiT}
\mypic{inline}{64}{\LoT}
\mypic{inline}{65}{\NiT}
\mypic{inline}{66}{\NoT}

\mypic{inline}{67}{\singleB}
\mypic{inline}{68}{\singleW}

\mypic{inline}{69}{\ThreeTwistTangle}
\mypic{inline}{70}{\ThreeTwistTangleBlue}

\mypic{inline}{71}{\DotCarc}
\mypic{inline}{72}{\DotBarc}
\mypic{inline}{73}{\TrivialTwoTangle}
\mypic{inline}{74}{\CircleDot}
\mypic{inline}{75}{\Nio}

\begin{document}
\title{Cosmetic operations and Khovanov multicurves}

\author{Artem Kotelskiy}
\address{Department of Mathematics \\ Stony Brook University}
\email{artem.kotelskiy@stonybrook.edu}

\author{Tye Lidman}
\address{Department of Mathematics \\ North Carolina State University}
\email{tlid@math.ncsu.edu}

\author{Allison H. Moore}
\address{Department of Mathematics \& Applied Mathematics \\ Virginia Commonwealth University}
\email{moorea14@vcu.edu}

\author{Liam Watson}
\address{Department of Mathematics \\ University of British Columbia}
\email{liam@math.ubc.ca}

\author{Claudius Zibrowius}
\address{Department of Mathematics \\ University of Regensburg}
\email{claudius.zibrowius@posteo.net}

\thanks{TL was partially supported by NSF grant DMS-1709702 and a Sloan fellowship. 
AK is supported by an AMS-Simons travel grant. 
AHM is supported by The Thomas F. and Kate Miller Jeffress Memorial Trust, Bank of America, Trustee. 
LW is supported by an NSERC discovery/accelerator grant.
CZ is supported by the Emmy Noether Programme of the DFG, Project number 412851057.}

\begin{abstract}
	We prove an equivariant version of the Cosmetic Surgery Conjecture for strongly invertible knots. 
	Our proof combines a recent result of Hanselman with the Khovanov multicurve invariants \(\Khr\) and \(\BNr\). 
	We apply the same techniques to reprove a result of Wang about the Cosmetic Crossing Conjecture and split links. Along the way, we show that \(\Khr\) and \(\BNr\) detect if a Conway tangle is split.
\end{abstract}
\maketitle
\setcounter{tocdepth}{1}


\section{Introduction}\label{sec:intro}

Here are two classical open conjectures in low dimensional topology:

\begin{csconj}
	Given a non-trivial knot \(K\subset S^3\) and \(r,r'\in\QPI\), suppose there exists an orientation-preserving diffeomorphism \(S^3_r(K)\cong S^3_{r'}(K)\). Then \(r=r'\). 
\end{csconj}

\begin{ccconj}
	Any crossing change that preserves the isotopy class of a knot must occur at a nugatory crossing, meaning the crossing circle (see Figure~\ref{fig:crossingcircle}) bounds an embedded disk in the complement of the knot.
\end{ccconj}

The Cosmetic Surgery Conjecture originates with Gordon \cite[Conjecture 6.1]{Gordon} and the Cosmetic Crossing Conjecture is due to Lin. 
Both problems appear in Kirby's Problem List \cite[Problem 1.81 A, Bleiler; Problem 1.58, Lin]{Kirby}.
This paper explores what Khovanov homology can say about these conjectures from the perspective of the multicurve technology developed by Kotelskiy, Watson, and Zibrowius \cite{KWZ}.

\begin{figure}[t]
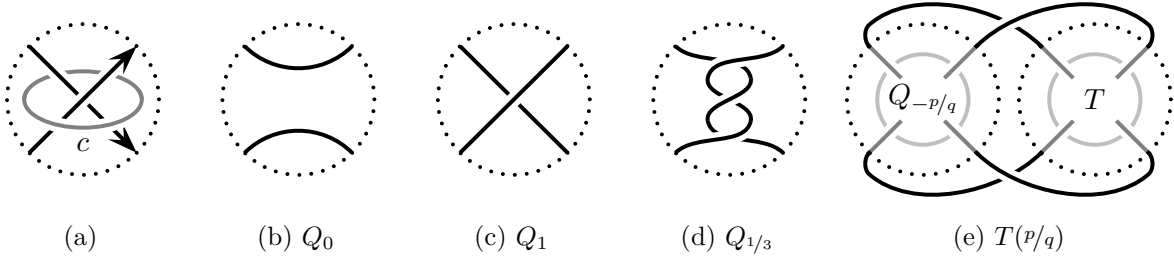

	\centering
	\begin{subfigure}{0.17\textwidth}
		\centering
		\(\crossingcircle\)
		\caption{}
		\label{fig:crossingcircle}
	\end{subfigure}
	\begin{subfigure}{0.17\textwidth}
		\centering
		\(\ratoxi\)
		\caption{\(Q_{0}\)}
		\label{fig:rat:oxi}
	\end{subfigure}
	\begin{subfigure}{0.17\textwidth}
		\centering
		\(\ratixi\)
		\caption{\(Q_{1}\)}
		\label{fig:rat:ixi}
	\end{subfigure}
	\begin{subfigure}{0.17\textwidth}
		\centering
		\(\ratixiii\)
		\caption{\(Q_{\nicefrac{1}{3}}\)}
		\label{fig:rat:ixiii}
	\end{subfigure}
	\begin{subfigure}{0.28\textwidth}
		\centering
		\(\tanglepairing\)
		\caption{\(T(\nicefrac{p}{q})\)}
		\label{fig:rat:pairing}
	\end{subfigure}
	\caption{A crossing circle \(c\) (a), some examples of rational tangles~(b--d), and the \(\nicefrac{p}{q}\)-rational filling \(T(\nicefrac{p}{q})\) of a Conway tangle \(T\)~(e)}
	\label{fig:rat}
\end{figure}

\subsection{Cosmetic surgeries}

Following Sakuma \cite{Sakuma}, a strongly invertible knot is a pair \((K,h)\), where \(K\) is a knot in \(S^3\) and \(h\) is an orientation-preserving involution of \(S^3\) mapping \(K\) to itself and reversing a choice of orientation of \(K\). 
Strongly invertible knots $(K,h)$ and $(K',h')$ are equivalent if there exists an orientation-preserving homeomorphism $f$ on $S^3$ for which $f(K)=K'$ (so that $K$ and $K'$ are equivalent knots) and $h = f^{-1}\circ h'\circ f$. 
Any strong inversion \(h\) on a knot \(K\subset S^3\) restricts to the hyperelliptic involution of the torus on the boundary of the knot exterior, and hence can be extended to an involution \(h_r\) on  \(S^3_r(K)\) for any slope \(r\in\QPI\).
(The latter statement can be found in \cite{Holzmann}. See also \cite{Auckly} for a nice exposition.) 
This extension is unique up to homotopy.

In this article we prove the following equivariant version of the Cosmetic Surgery Conjecture:

\begin{theorem}\label{thm:intro:ecsc}
	Given a non-trivial strongly invertible knot \((K,h)\) and \(r,r'\in\QPI\), suppose that there exists an orientation-preserving diffeomorphism \(f\co S^3_r(K)\rightarrow S^3_{r'}(K)\) such that $h_{r'}\circ f= f \circ h_r$. Then \(r=r'\).
\end{theorem}

A strongly invertible knot $(K,h)$ gives rise to a Conway (ie four-ended) tangle $T \subset B^3$ constructed as follows. According to the Smith conjecture, the fixed point set $\operatorname{Fix}(h)$ in $S^3$ is an unknot intersecting $K$ in two points, and thus restricts to a pair of arcs in the knot exterior $X_K=S^3\smallsetminus\nu(K)$. Therefore, taking the quotient produces a Conway tangle $T=\operatorname{Im}(\operatorname{Fix}(h)\cap X_K)$ inside the three ball $B^3=X_K / h$. For an illustration, see \cite[Figure~4]{Watson2012}.

Cosmetic surgeries along a strongly invertible knot  $(K,h)$ are closely related to cosmetic tangle fillings of the quotient tangle $T$, due to the Montesinos trick \cite{Montesinos}:
The fixed point set of the involution $h_{\nicefrac{p}{q}}$ on $S^3_{\nicefrac{p}{q}}(K)$ restricted to the surgery solid torus is a pair of arcs, which descends in the quotient of the solid torus to a trivial Conway tangle. Montesinos gives an explicit correspondence between the rational surgery slope and a rational parameterization of the trivial tangle in the quotient. In particular, $S^3_{\nicefrac{p}{q}}(K)$ is the two-fold branched cover of the rational tangle filling $T(\nicefrac{p}{q}) = Q_{\nicefrac{-p}{q}}\cup T$  illustrated in Figure~\ref{fig:rat}, where $Q_{\nicefrac{-p}{q}}$ denotes the rational tangle of slope $\nicefrac{-p}{q}$ and the tangle $T$ is the quotient tangle of $(K,h)$. 
Conversely, a cosmetic crossing change on a knot induces a cosmetic surgery on the two-fold branched cover.  Hence, the two cosmetic conjectures are related, and the analogue of Theorem~\ref{thm:intro:ecsc} is:

\begin{theorem}\label{thm:intro:cosmeticrational}
	Let \(T\) be a Conway tangle in the three-ball with an unknot closure.  Suppose \(T(r)\) and \(T(r')\) are isotopic links in the three-sphere for some \(r, r' \in \QPI\).  Then \(r = r'\) or \(T\) is rational.   
\end{theorem}

Indeed Theorem~\ref{thm:intro:cosmeticrational} is equivalent to Theorem~\ref{thm:intro:ecsc} by taking two-fold branched covers.  The condition of having an unknot closure is equivalent to considering surgeries on knots in $S^3$ and non-triviality of the knot in $S^3$ is equivalent to the tangle not being rational.  

Gordon and Luecke's solution to the knot complement problem implies that no non-trivial knots admit cosmetic surgeries when one of the slopes is $\infty$ \cite{GordonLuecke}. Boyer and Lines showed that if $\Delta''_K(t)\neq0$, rational surgeries along $K$ are always distinct, where $\Delta_K(t)$ is the Alexander polynomial \cite[Proposition 5.1]{BoyerLines}, and a similar result may be formulated in terms of the Jones polynomial \cite{IchiharaWu}. Note that the obstructions of \cite{BoyerLines, IchiharaWu} do not disqualify strongly invertible knots from admitting cosmetic surgeries because all Alexander polynomials of knots can be realized by strongly invertible knots \cite{Sakai}. 

Strong restrictions on cosmetic surgeries can be formulated in terms of Heegaard Floer homology, as shown in work of Wang \cite{WangJiajun}, Ozsv\'ath and Szab\'o \cite{OzSzRational}, and Ni and Wu \cite{NiWu}.  
Using immersed curves, Hanselman's work \cite{HanselmanCSC} extends these results; in particular, he shows that cosmetic surgery slopes have to be either \(\pm2\) or \(\nicefrac{\pm1}{n}\). 
In an entirely different direction, hyperbolic geometry techniques have been used to bound the lengths of cosmetic filling slopes \cite{FuterPurcellSchleimer}. 
Using the strategies above and others, the conjecture has been established for knots of genus one, cables, connected sums and three-braids \cite{WangJiajun, Tao1, Tao2, Varvarezos}. 

Absent from this summary are developments utilizing Khovanov homology. 
Our main tool for approaching this problem is the multicurve technology for Khovanov homology developed by Kotelskiy, Watson, and Zibrowius \cite{KWZ,KWZthinness}, which we use to prove Theorem~\ref{thm:intro:cosmeticrational}.  
This theory assigns an immersed multicurve in the four-punctured sphere to a Conway tangle and a suitable Lagrangian Floer homology of these multicurves computes the Khovanov homology.

\subsection{Cosmetic Crossings}
As illustrated by Theorem~\ref{thm:intro:cosmeticrational}, the immersed curve invariants for Khovanov homology provide a powerful tool for studying the behavior of Khovanov homology under tangle fillings.  Note that the Cosmetic Crossing Conjecture can be viewed in terms of comparing the $+1$ and $-1$ tangle fillings of a Conway tangle; hence, the Khovanov tangle invariants provide a natural tool for studying the Cosmetic Crossing Conjecture.  In fact,  Theorem~\ref{thm:intro:cosmeticrational} immediately implies the Cosmetic Crossing Conjecture for the unknot, originally due to Scharlemann-Thompson \cite{ST}.  This conjecture is still open in general, but has been established for knots which are two-bridge \cite{Torisu} or fibered \cite{Kalfagianni, Rogers}, as well as for large classes of knots which are genus one \cite{BFKP, Ito} or alternating \cite{LidmanMoore}.  In the second half of this paper, we illustrate the utility of this theory by giving elementary proofs of two other known results about the Cosmetic Crossing Conjecture.

First, we recall that there is a generalization of the Cosmetic Crossing Conjecture.  Let $c$ be a crossing circle for a knot $K$ as in Figure~\ref{fig:crossingcircle}.  
Then, performing $\nicefrac{-1}{n}$-Dehn surgery on $c$ produces a new knot $K_n$ which corresponds to adding $n$ full right-handed twists at the crossing.  The Generalized Cosmetic Crossing Conjecture predicts that if $c$ is a non-nugatory crossing, then $K_n$ is not isotopic to $K_m$ for $n \neq m$.  

We first give a new proof of a recent result of Wang on the Generalized Cosmetic Crossing Conjecture, which also used Khovanov homology:
\begin{theorem}[Wang \cite{Wang}] \label{thm:josh}
Let $K$ be a knot obtained by a non-trivial band surgery on a split link $L$.  If $K_n$ is obtained by inserting $n \in \Z$ twists into the band, then $K_n$ is not isotopic to $K_m$ for any $n \neq m$.
\end{theorem}

We also prove that the generalized crossing conjecture holds ``asymptotically'':
\begin{theorem}\label{thm:agccc}
Let $K$ be an unoriented knot and $c$ a crossing circle for a non-nugatory crossing. Let $\{K_n\}_{n \in \Z}$ be the associated sequence of knots obtained by inserting twists at $c$. Then there exists an integer $N$ such that $\{K_n\}_{|n|\geq N}$ are pairwise different.
\end{theorem}

\begin{remark}
While we could not find Theorem~\ref{thm:agccc} written explicitly in the literature, it is certainly known to experts using standard techniques from three-manifold topology.  Rather than give this alternative proof here in full, we illustrate this by sketching a proof for a suitably generic case.  Suppose that $K \cup c$ is a hyperbolic link.  Then $K_n$ is obtained by performing $\nicefrac{-1}{n}$-surgery on $c$.  Hence, $K_n$ is hyperbolic for all but finitely many $n$ and further the hyperbolic volume of $K_n$ converges to the hyperbolic volume of $K \cup c$, which is strictly greater than that of any $K_n$.  The asymptotics of this convergence is described by the work of Neumann-Zagier \cite{NeumannZagier} and precludes having more than finitely many knots in the sequence of fixed volume. 
It is interesting that Khovanov homology and the hyperbolic volume establish the same result in this setting. 
We note that the behaviour of the Jones polynomial under twisting in relation with hyperbolic geometry has been considered, see \cite{CK} for example.
\end{remark}

Along the way, we establish the following technical result about the Khovanov multicurve invariants, which may be of independent interest:

\begin{theorem}\label{thm:intro:split}
	The invariants \(\Khr(T)\) and \(\BNr(T)\) detect if the Conway tangle \(T\) is split. 
\end{theorem}

Note that Theorem~\ref{thm:intro:split} has been established for the analogous knot Floer homology multicurve invariant by Lidman, Moore, and Zibrowius \cite{LMZ}.  

\subsection*{Outline}
In Section~\ref{sec:background}, we give the requisite background on the immersed curve invariants for tangles.  In Section~\ref{sec:split_tangle_detection}, we prove Theorem~\ref{thm:intro:split}.  In Section~\ref{sec:ecsc}, we prove Theorem~\ref{thm:intro:cosmeticrational} (and hence Theorem~\ref{thm:intro:ecsc}).  In Section~\ref{sec:agccc}, we prove Theorem~\ref{thm:agccc} and Theorem~\ref{thm:josh}.  

\subsection*{Acknowledgements}
	We thank Joshua Wang for useful and motivating discussions around the Generalized Cosmetic Crossing Conjecture, and in particular for posing the main question that we answer in Section~\ref{sec:agccc}.
\pagebreak
\section{Review of the Khovanov multicurve invariants}\label{sec:background}

In this section, we review some properties of the immersed curve invariants \(\Khr\) and  \(\BNr\) of pointed Conway tangles from~\cite{KWZ}. 
We work exclusively over the field \(\fieldTwoElements\) of two elements and only summarize those properties that we will need in this paper; more elaborate introductions highlighting different aspects of the invariants can be found in~\cite{KWZ-strong, KWZthinness}.

Let \(T\) be an oriented \emph{pointed} Conway tangle, that is a four-ended tangle in the three-ball \(B^3\) with a choice of distinguished tangle end, which we mark by~\(\ast\). 
Denote by \(\FourPuncturedSphereKh\) the four-punctured sphere \(\partial B^3\smallsetminus \partial T\); the puncture marked by \(\ast\) will be called \emph{special}. 
We associate with such a tangle \(T\) invariants \(\BNr(T)\) and \(\Khr(T)\) that take the form of multicurves on \(\FourPuncturedSphereKh\). 
By multicurve, we mean a collection of immersed curves that carry certain extra data. 
Broadly speaking, there are two kinds of such curves: compact and non-compact. 
A compact immersed curve in \(\FourPuncturedSphereKh\) is an immersion of \(S^1\), considered up to regular homotopy, that (up to conjugation) defines a primitive element of \(\pi_1(\FourPuncturedSphereKh)\), and 
each of these curves is decorated with a local system, ie an invertible matrix over \(\fieldTwoElements\) 
considered up to matrix similarity. 
A non-compact immersed curve in \(\FourPuncturedSphereKh\) is a non-null-homotopic immersion of an interval, with ends on the three non-special punctures of \(\FourPuncturedSphereKh\); see \cite[Definition~1.4]{KWZ}. 
Non-compact curves do not carry local systems. 
In addition, all curves are equipped with a bigrading; more on this in Section~\ref{subsec:background:bigrading} below. 

\begin{remark}\label{rem:Khr:2m3pt}
	We often draw \(\FourPuncturedSphereKh\) as the plane plus a point at infinity minus the four punctures. To help identify this abstract surface with \(\partial B^3\smallsetminus \partial T\), we then add two dotted gray arcs that parametrize the surface, see   
	Figures~\ref{fig:Kh:example:tangle}--\ref{fig:Kh:example:Curve:Downstairs}. 
	The blue curves in these figures show the multicurves \(\BNr(P_{2,-3})\) and \(\Khr(P_{2,-3})\) for the pretzel tangle \(P_{2,-3}\);  cf~\cite[Example~6.7]{KWZ}.
	All components of these curves carry the (unique) one-dimensional local system.
\end{remark}

\begin{figure}[b]
	\centering
	\begin{subfigure}{0.30\textwidth}
		\centering
		\(\pretzeltangleKh\)
		\caption{The pretzel tangle \(P_{2,-3}\)}\label{fig:Kh:example:tangle}
	\end{subfigure}
	\begin{subfigure}{0.28\textwidth}
		\centering
		\(\pretzeltangleDownstairsBNr\)
		\caption{\(\BNr(P_{2,-3})\)}\label{fig:BN:example:Curve:Downstairs}
	\end{subfigure}
	\begin{subfigure}{0.28\textwidth}
		\centering
		\(\pretzeltangleDownstairsKhr\)
		\caption{\(\Khr(P_{2,-3})\)}\label{fig:Kh:example:Curve:Downstairs}
	\end{subfigure}
	\bigskip
	\\
	\begin{subfigure}{0.42\textwidth}
		\centering
		\(\pretzeltangleUpstairsBNr\)
		\caption{A lift of \(\BNr(P_{2,-3})\) to \(\PuncturedPlane\) }\label{fig:BN:example:Curve:Upstairs}
	\end{subfigure}
	\begin{subfigure}{0.55\textwidth}
		\centering
		\(\pretzeltangleUpstairsKhr\)
		\caption{A lift of \(\Khr(P_{2,-3})\) to \(\PuncturedPlane\) }\label{fig:Kh:example:Curve:Upstairs}
	\end{subfigure}
	\caption{The multicurve invariants for the pretzel tangle \(P_{2,-3}\). 
	Under the covering \(\PuncturedPlane\rightarrow\FourPuncturedSphereKh\), the shaded regions in (b+c) correspond to the shaded regions in (d+e). 
 }\label{fig:Kh:example}
\end{figure} 

\subsection{The construction of the multicurves}\label{subsec:background:construction} The starting point is the algebraic tangle invariant \(\KhTl{T}\) due to Bar-Natan.  The invariant
\(\KhTl{T}\) is a chain complex over a certain cobordism category, whose objects are crossingless tangle diagrams \cite{BarNatanKhT};  we refer to~\cite[Section~2]{KWZ} for a detailed introduction to complexes over categories/algebras, and the equivalent viewpoint through type~D structures.
In \cite[Theorem~1.1]{KWZ}, it was shown that any such complex can be rewritten as a chain complex \(\DD(T)\) over the following  category $\BNAlgH$, consisting of two objects and morphisms equal to paths in a quiver modulo relations: 
\begin{equation*}\label{eq:B_quiver}
\BNAlgH
\coloneqq
\fieldTwoElements\Big[
\begin{tikzcd}[row sep=2cm, column sep=1.5cm]
\DotB
\arrow[leftarrow,in=145, out=-145,looseness=5]{rl}[description]{D_{\bullet}}
\arrow[leftarrow,bend left]{r}[description]{S_{\circ}}
&
\DotC
\arrow[leftarrow,bend left]{l}[description]{S_{\bullet}}
\arrow[leftarrow,in=35, out=-35,looseness=5]{rl}[description]{D_{\circ}}
\end{tikzcd}
\Big]\Big/\Big(
\parbox[c]{90pt}{\footnotesize\centering
	$D_{\bullet} \cdot S_{\circ}=0=S_{\circ}\cdot D_{\circ}$\\
	$D_{\circ}\cdot S_{\bullet}=0=S_{\bullet}\cdot D_{\bullet}$
}\Big)
\end{equation*}
Here, the objects \(\DotB\) and \(\DotC\) correspond to the crossingless tangles \(\Lo\) and \(\Li\), respectively. 

We will refer to $\BNAlgH$ as a (quiver) algebra, and to \(\DD(T)\) as a chain complex over the algebra $\BNAlgH$.
Defining \(D\coloneqq D_{\bullet} + D_{\circ}\) and \(S\coloneqq S_{\bullet} + S_{\circ}\) often allows us to drop the subscripts of the algebra elements of \(\BNAlgH\).
The chain homotopy type of \(\DD(T)\) is an invariant of the tangle~\(T\). 
Moreover, using the central element 
\[
H\coloneqq D+S^2 = D_{\bullet} + D_{\circ} + S_{\circ}S_{\bullet} + S_{\bullet}S_{\circ} ~ \in ~ \BNAlgH 
\]   
we define a chain complex \(\DD_1(T)\) as the mapping cone 
\[\DD_1(T)\coloneqq \Big[q^{-1}h^{-1}\DD(T)\xrightarrow{H\cdot \id} q^{1}h^0\DD(T)\Big]\]
where $H\cdot \id$ is the endomorphism of $\DD(T)$ defined by $x\xrightarrow{H}x$ for all generators $x$ of $\DD(T)$.
The chain homotopy type of \(\DD_1(T)\)  is also a tangle invariant.

\begin{figure}[b]
	\centering
	\begin{subfigure}[t]{0.29\textwidth}
		\centering
		$\PairingTrefoilLoopINTRO$
		\caption{$\textcolor{red}{\gamma\hateqq
				[
				\protect\begin{tikzcd}[nodes={inner sep=2pt}, column sep=23pt,ampersand replacement = \&]
				\protect\DotCred
				\protect\arrow{r}{D+S^2}
				\protect\&
				\protect\DotCred
				\protect\end{tikzcd}
				]}$}\label{fig:exa:classification:curves:loop}
	\end{subfigure}
	\begin{subfigure}[t]{0.29\textwidth}
		\centering
		$\PairingTrefoilArcINTRO$
		\caption{$\textcolor{blue}{\gamma'\hateqq
				[
				\protect\begin{tikzcd}[nodes={inner sep=2pt}, column sep=13pt,ampersand replacement = \&]
				\protect\DotCblue
				\protect\arrow{r}{S}
				\protect\&
				\protect\DotBblue
				\protect\arrow{r}{D}
				\protect\&
				\protect\DotBblue
				\protect\arrow{r}{S^2}
				\protect\&
				\protect\DotBblue
				\protect\end{tikzcd}
				]}$}\label{fig:exa:classification:curves:arc}
	\end{subfigure}
	\begin{subfigure}[t]{0.29\textwidth}
		\centering
		$\PairingBothINTRO$
		\vspace{5pt}
		\caption{\(\HF(\textcolor{red}{\gamma},\textcolor{blue}{\gamma'})\cong\fieldTwoElements^3\)}\label{fig:exa:classification:curves:pairing}
	\end{subfigure}
	\caption{Two immersed curves and their corresponding chain complexes (a+b) and their Lagrangian Floer homology (c); cf \cite[Examples~1.6 and~1.7]{KWZ}
	}\label{fig:exa:classification:curves}
\end{figure}
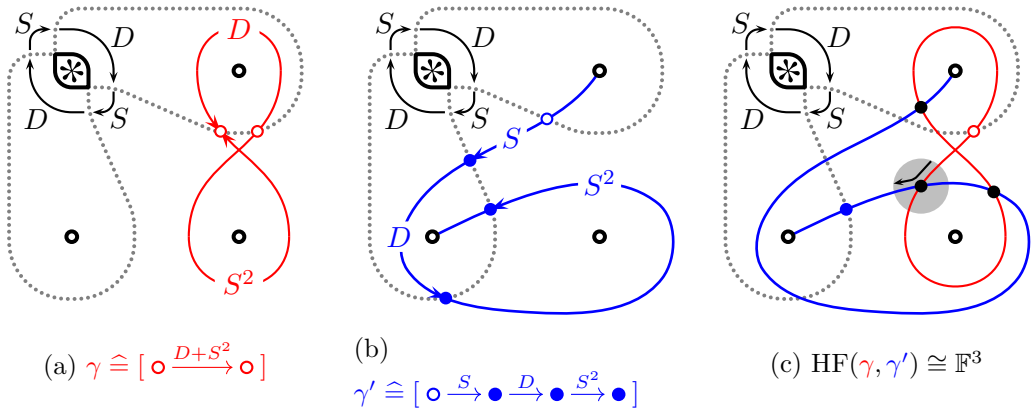

The multicurve invariants \(\BNr(T)\) and \(\Khr(T)\) are geometric interpretations of $\DD(T)$ and $\DD_1(T)$ respectively, made possible by the following classification result: 
The homotopy equivalence classes of chain complexes over \(\BNAlgH\) are in one-to-one correspondence with multicurves on the four-punctured sphere \(\FourPuncturedSphereKh\) \cite[Theorem~1.5]{KWZ}. 
In a little more detail, this correspondence (which we denote by $\hateqq$) uses the parametrization of \(\FourPuncturedSphereKh\) given by the two dotted arcs described in Remark~\ref{rem:Khr:2m3pt}. 
We will generally assume that the multicurves intersect these arcs minimally. Then, roughly speaking, the intersection points correspond to generators of the associated chain complexes and paths between those intersection points correspond to the differentials. 
The two examples in Figure~\ref{fig:exa:classification:curves} should give the reader a general impression how this works.

\begin{example}\label{exa:rational_tangles}
	For the trivial tangle \(Q_\infty=\Li\), the chain complex \(\DD(Q_\infty)\) consists of a single object \(\DotC\) and the differential vanishes. The corresponding multicurve \(\BNr(Q_{\infty})\) consists of a single vertical arc connecting the two non-special tangle ends. The chain complex \(\DD_1(Q_\infty)\) and the corresponding curve \(\Khr(Q_\infty)\) is shown in Figure~\ref{fig:exa:classification:curves:loop}.	The local system on this curve is one-dimensional.
	
	The tangle \(Q_{\nicefrac{1}{3}}=\ThreeTwistTangle\) is obtained from the trivial tangle \(Q_\infty\) by adding three twists to the two lower tangle ends. Its invariant \(\BNr(Q_{\nicefrac{1}{3}})\) is shown in Figure~\ref{fig:exa:classification:curves:arc}. Note that it agrees with the vertical arc \(\BNr(Q_{\infty})\) up to three twists. This is not a coincidence; one can show that adding twists to \emph{any} tangle (not just a rational tangle) corresponds to adding twists to the multicurves; see \cite[Theorem~1.13]{KWZ}.
	Thus, the identification of \(\partial B^3\smallsetminus \partial T\) with the abstract surface~\(\FourPuncturedSphereKh\) containing the multicurves is natural.
\end{example}

\subsection{A gluing theorem}\label{subsec:background:gluing}

\begin{figure}[t]
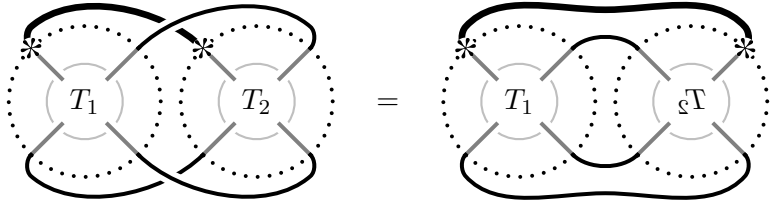

	\centering
	\(
	\tanglepairingI
	\quad = \quad
	\tanglepairingII
	\)
	\caption{Two tangle decompositions defining the link \(T_1\cup T_2\). The tangle \protect\reflectbox{\(T_2\)} is the result of rotating \(T_2\) around the vertical axis. By rotating the entire link on the right-hand side around the vertical axis, we can see that \(T_1\cup T_2=T_2\cup T_1\).}
	\label{fig:tanglepairing} 
\end{figure}

The multicurve invariants satisfy various gluing formulas \cite[Theorem~1.9]{KWZ}. The one that we will use in this paper is the following:

\begin{theorem}\label{thm:GlueingTheorem:Kh}
	Let \(L=T_1\cup T_2\) be the result of gluing two oriented pointed Conway tangles as in Figure~\ref{fig:tanglepairing} such that the orientations match. 
	Let \(\mirror\)
	be the map identifying the two four-punctured spheres. 
	Then  
	\[
	\Khr(L)
	\cong 
	\HF\left(\mirror(\Khr(T_1)),\BNr(T_2)\right)
	\cong
	\HF\left(\mirror(\BNr(T_1)),\Khr(T_2)\right)
	\]
\end{theorem}

The Lagrangian Floer homology \(\HF(\textcolor{red}{\gamma},\textcolor{blue}{\gamma'})\) between two curves \(\textcolor{red}{\gamma}\) and \(\textcolor{blue}{\gamma'}\) is a vector space that can be computed as follows. First, we draw the curves in such a way that minimizes the number of intersection points between \(\textcolor{red}{\gamma}\) and \(\textcolor{blue}{\gamma'}\). 
\(\HF(\textcolor{red}{\gamma},\textcolor{blue}{\gamma'})\) is then equal to a vector space freely generated by those intersection points, provided that the curves are not homotopic to each other \cite[Theorem~5.25]{KWZ}. 
(We will always be able to make this assumption in this paper.)
For instance, with Example~\ref{exa:rational_tangles} and Figure~\ref{fig:exa:classification:curves} in mind, the Khovanov homology of the trefoil can be computed as follows:
$$\Khr(\Li \cup \ThreeTwistTangle) \cong \HF( \Khr(\Li), \BNr(\ThreeTwistTangle)=\F^3$$
Finally, the Lagrangian Floer homology between two multicurves is simply the direct sum of the Lagrangian Floer homologies between individual components.

\subsection{Gradings}\label{subsec:background:bigrading}
Khovanov homology is a \emph{bigraded} homology theory, and this bigrading is often what makes it a powerful invariant. 
The multicurves \(\BNr(T)\) and \(\Khr(T)\) also carry bigradings. We now describe how the gradings work: first on the algebra $\BNAlgH$, then on chain complexes, then on multicurves, and then finally on Lagrangian Floer homology between multicurves. 

Equip the algebra \(\BNAlgH\) with  quantum grading \(q\), which is determined by 
\[
q(D_{\bullet}) = q(D_{\circ}) = -2
\qquad 
\text{and}
\qquad
q(S_{\bullet}) = q(S_{\circ}) = -1
\] 
The homological grading is defined to be $0$ for all elements of \(\BNAlgH\).

Differentials of bigraded chain complexes over \(\BNAlgH\) are required to preserve quantum grading and increase the homological grading by $1$. Concretely this means that if a differential contains a morphism $x \xrightarrow{a}y$ (where $x$ and $y$ are generators of the complex and $a\in\BNAlgH$), then $q(a)+q(y)-q(x)=0$ and $h(a)+h(y)-h(x)=1$. 
We often specify the quantum gradings \(\QGrad{q}\) of generators of such complexes via superscripts, like so: \(\GGzqh{x}{}{q}{}\).

The bigrading on a multicurve takes the form of a bigrading on the intersection points between the multicurve and the two parametrizing arcs, ie the generators of the corresponding chain complex over \(\BNAlgH\). 
If \(T\) is an unoriented tangle, the bigradings on \(\BNr(T)\) and \(\Khr(T)\) are only well-defined as relative bigradings; to fix the overall shift, an orientation of \(T\) is required. 

Let \(\textcolor{red}{\gamma}\) and \(\textcolor{blue}{\gamma'}\) be two bigraded multicurves and suppose \(\textcolor{red}{X\hateqq\gamma}\) and \(\textcolor{blue}{X'\hateqq\gamma'}\) are the corresponding bigraded chain complexes over \(\BNAlgH\). 
Then \(\HF(\textcolor{red}{\gamma},\textcolor{blue}{\gamma'})\) also carries a bigrading. 
It can be computed using the fact that this vector space is bigraded isomorphic to the homology of the morphism space \(\Mor(\textcolor{red}{X},\textcolor{blue}{X'})\) \cite[Theorem~1.5]{KWZ} (see the discussion before~\cite[Definition~2.4]{KWZ} for the definition of the differential on a morphism space between two complexes). 
The quantum grading of a morphism 
\(
\begin{tikzcd}[nodes={inner sep=2pt}, column sep=13pt,ampersand replacement = \&]
\GGzqh{x}{}{b}{}
\arrow{r}{\alpha}
\&
\GGzqh{y}{}{a}{}
\end{tikzcd}
\), 
where \(\alpha\in\BNAlgH\), is computed using the formula \(\QGrad{a}-\QGrad{b}+q(\alpha)\); an analogous formula holds for the homological grading. 
Each intersection point generating \(\HF(\textcolor{red}{\gamma},\textcolor{blue}{\gamma'})\) corresponds to a morphism from which we can read off the bigrading \cite[Section~7]{KWZ-strong}. For instance, the highlighted intersection point in Figure~\ref{fig:exa:classification:curves:pairing} corresponds to the morphism \(
\begin{tikzcd}[nodes={inner sep=2pt}, column sep=13pt,ampersand replacement = \&]
\DotCred
\arrow{r}{S}
\&
\DotBblue
\end{tikzcd}
\), 
so the bigrading of this intersection point is equal to 
\[
q(\DotB) = q(\DotBblue)-q(\DotCred)+q(S)
\quad
\text{ and }
\quad
h(\DotB) = h(\DotBblue)-h(\DotCred)
\]

\subsection{\texorpdfstring{Geography of components of \(\Khr\)}{Geography of components of Khr}}\label{subsec:background:geography} 

We now recall some basic facts about $\Khr(T)$ and $\BNr(T)$ from \cite[Section~6]{KWZ}.
In this paper, we will focus only on tangles without closed components. For such tangles, 
\(\BNr(T)\) consists of a single non-compact component and a (possibly zero) number of compact components.  
In contrast, $\Khr(T)$ consists of compact components only. 

Often, multicurves become easier to manage when considered in a certain covering space of~\(\FourPuncturedSphereKh\), namely the planar cover that factors through the toroidal two-fold cover: 
$$
(\R^2 \smallsetminus \Z^2) \to (T^2 \smallsetminus 4\pt) \to \FourPuncturedSphereKh
$$
This is illustrated in Figure~\ref{fig:Kh:example} for the multicurve invariants of the pretzel tangle \(P_{2,-3}\).

\begin{figure}[t]
	\centering
	\begin{subfigure}{0.3\textwidth}
		\centering
		\labellist 
		\footnotesize \color{blue}
		\pinlabel $\sKh_{2}(0)$ at 65 115
		\pinlabel $\rKh_{1}(0)$ at 65 28
		\endlabellist
		\includegraphics[scale=1]{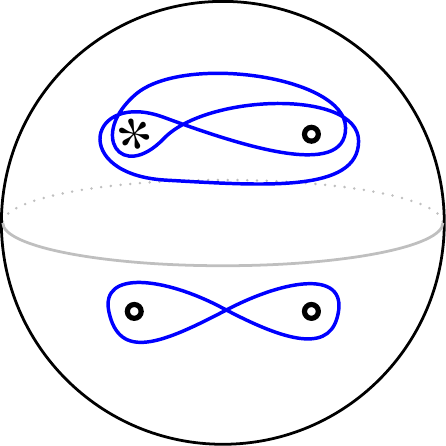}
		\caption{\(n=1\)}\label{fig:geography:Downstairs1}
	\end{subfigure}
	\begin{subfigure}{0.3\textwidth}
		\centering
		\labellist 
		\footnotesize \color{blue}
		\pinlabel $\sKh_{4}(0)$ at 65 118
		\pinlabel $\rKh_{2}(0)$ at 65 13
		\endlabellist
		\includegraphics[scale=1]{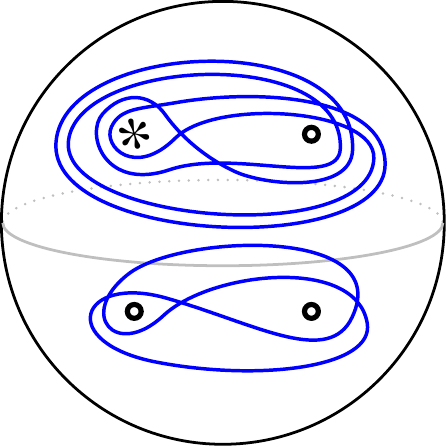}
		\caption{\(n=2\)}\label{fig:geography:Downstairs2}
	\end{subfigure}
	\begin{subfigure}{0.3\textwidth}
		\centering
		\labellist 
		\footnotesize \color{blue}
		\pinlabel $\sKh_{6}(0)$ at 65 121
		\pinlabel $\rKh_{3}(0)$ at 65 10
		\endlabellist
		\includegraphics[scale=1]{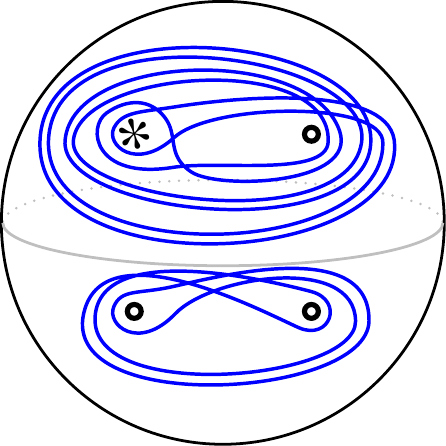}
		\caption{\(n=3\)}\label{fig:geography:Downstairs3}
	\end{subfigure}
	\\
	\begin{subfigure}{0.9\textwidth}
		\centering
		\(\GeographyCovering\)
		\caption{}\label{fig:geography:Upstairs}
	\end{subfigure}
	\caption{The curves \(\rKh_n(0)\) and \(\sKh_{2n}(0)\) (a--c) and their lifts to \(\PuncturedPlane\) (d). While not visually apparent, the curves \(\rKh_n(0)\) are invariant under the Dehn twist interchanging the lower two punctures. 
	}\label{fig:geography}
\end{figure}
 
\begin{definition}
	Given an immersed curve $c \looparrowright \FourPuncturedSphereKh$, denote by $\tc$ a lift of $\c$ to the cover \(\PuncturedPlane\). 
	For \(n\in\mathbb N\), let \(\rKh_{n}(0)\) and \(\sKh_{2n}(0)\) be the immersed curves in \(\FourPuncturedSphereKh\) that respectively admit lifts to the curves $\tr_{n}(0)$ and $\ts_{2n}(0)$ in Figure~\ref{fig:geography:Upstairs}; curves for $n=1,2,3$ are illustrated in Figures~\ref{fig:geography:Downstairs1}--\ref{fig:geography:Downstairs3}.
	For every \(\nicefrac{p}{q}\in\QPI\), we respectively define the curves $\rKh_n(\nicefrac{p}{q})$ and $\sKh_{2n}(\nicefrac{p}{q})$ as the images of \(\rKh_{n}(0)\) and \(\sKh_{2n}(0)\) under the action of
	\[
	\begin{bmatrix*}[c]
	q & r \\
	p & s
	\end{bmatrix*}
	\]
	considered as an element the mapping class group fixing the special puncture $\Mod(\FourPuncturedSphereKh) \cong \PSL(2,\Z)$, where \(qs-pr=1\). 
	(This transformation maps straight lines of slope 0 to straight lines of slope \(\nicefrac{p}{q}\).)
	We call \(\rKh_{n}(\nicefrac{p}{q})\) a curve of \emph{rational type, slope \(\nicefrac{p}{q}\), and length \(n\)}.    
	We call \(\sKh_{2n}(\nicefrac{p}{q})\) a curve of \emph{special type, slope \(\nicefrac{p}{q}\), and length \(2n\)}. 
	The local systems on all these curves are defined to be trivial. 
\end{definition}

The following classification result is \cite[Theorem~6.5]{KWZthinness}. 

\begin{theorem}
	For any pointed Conway tangle \(T\), every component of \(\Khr(T)\) is equal to \(\rKh_n(\nicefrac{p}{q})\) or \(\sKh_{2n}(\nicefrac{p}{q})\) for some \(n\in\mathbb N\) and \(\nicefrac{p}{q}\in\QPI\), up to some bigrading shift. 
	In other words, components of \(\Khr(T)\) are completely classified by their type, slope, length, and bigrading. 
\end{theorem}

	As already mentioned in Example~\ref{exa:rational_tangles}, the multicurve invariants are natural with respect to adding twists; these twists generate $\Mod(\FourPuncturedSphereKh)$. 
	Thus, the invariant \(\Khr(Q_{\nicefrac{p}{q}})\) of a \(\nicefrac{p}{q}\)-rational tangle \(Q_{\nicefrac{p}{q}}\) is equal to \(\rKh_1(\nicefrac{p}{q})\), justifying the terminology. 
	In fact, we have the following detection result \cite[Theorem~5.7]{KWZthinness}.
	
	\begin{theorem}\label{thm:rational_tangle_detection}
		A pointed Conway tangle \(T\) is rational if and only if \(\Khr(T)\) consists of a single component \(\rKh_1(\nicefrac{p}{q})\) for some \(\nicefrac{p}{q}\in\QPI\). 
	\end{theorem}
	
	Rational components can also occur in the invariants of non-rational tangles. In fact, if \(T\) has no closed component, we know that there is always at least one such component \cite[Corollary~6.42]{KWZthinness}.  For example, the curve $\Khr(P_{2,-3})$ from Figure~\ref{fig:Kh:example:Curve:Downstairs} consists of the special component $\s_4(0)$ and the rational component $\r_1(\nicefrac{1}{2})$.
	Such rational components detect how tangle ends are connected \cite[Theorem~6.41]{KWZthinness}:
	
	\begin{theorem}\label{thm:detection:connectivity}
		Suppose a pointed Conway tangle \(T\) has connectivity \(\No\). Then the slope \(\nicefrac{p}{q}\in\QPI\) of any odd-length rational component of \(\Khr(T)\) satisfies \(p\equiv0\mod 2\). 
	\end{theorem}
	
	\begin{remark}
		Components of the invariant $\BNr(T)$, even compact ones, can be much more complicated than components of \(\Khr(T)\); for an example, see \cite[Figure~26]{KWZthinness}. However, the invariants of rational tangles are very simple: Any lift of \(\BNr(Q_{\nicefrac{p}{q}})\) to \(\PuncturedPlane\) is a straight line segment of slope \(\nicefrac{p}{q}\) connecting the lifts of two non-special punctures. 
	\end{remark}

\subsection{A dimension formula}

\begin{definition}
	Given two slopes \(\nicefrac{p}{q},\nicefrac{p'}{q'}\in\QPI\), where \((p,q)\) and \((p',q')\) are pairs of mutually prime integers, define the distance between two slopes as
	\[
	\Delta(\nicefrac{p}{q},\nicefrac{p'}{q'}) 
	\coloneqq
	\left|
	\det
	\begin{bmatrix}
	q & q' \\ 
	p & p' 
	\end{bmatrix}
	\right|
	=
	|qp'-pq'|
	\]
\end{definition}

\begin{lemma}\label{lem:pairing_linear_curves:dimension_formula}
	Let \(s,r\in\QPI\) be two distinct slopes. 
	Let \(\a_{s}\coloneqq\BNr(Q_{s})\) and 
	let \(\gamma\) be a rational or special curve of length \(\ell\) and slope \(r\). Then 
	\[
	\dim\HF(\a_{s},\gamma)
	=
	\ell\cdot\Delta(s,r)
	\]
\end{lemma}

\begin{proof}
	Write \(s=\nicefrac{p}{q}\) and \(r=\nicefrac{p'}{q'}\) for pairs \((p,q)\) and \((p',q')\) of mutually prime integers. 
	Observe that \(\dim\HF(\a_{s},\gamma)\) stays invariant under changing the parametrization of the four-punctured sphere. 
	So let us apply the linear transformation corresponding to the matrix 
	\[
	\begin{bmatrix}
	n & -m\\
	p' & -q'
	\end{bmatrix}
	\]
	where \(n,m\in\Z\) are such that \(mp'-nq'=1\). 
	This transformation maps \(\gamma\) to a curve of slope \(0\) and \(\a_s\) to a curve of slope \(\tfrac{p'q-q'p}{nq-mp}\). The distance between these curves remains the same. 
	This shows that if the formula holds for the case \(r=0\), then it also holds in general. 
	So suppose \(r=0\). In this case, \(\Delta(s,r)=|p|\), so we need to see that 
		\[
	\dim\HF(\a_s,\gamma)
	=
	\ell\cdot|p|
	\]
	This can be easily checked in the covering space \(\PuncturedPlane\). 
	If \(\tilde{\a}_s\) is a lift of \(\a_s\), \(\dim\HF(\a_s,\gamma)\) is equal to the number of times that \(\tilde{\a}_s\) intersects the preimage of \(\gamma\), which, by inspection, is indeed equal to \(\ell\cdot|p|\). 
\end{proof}

\begin{remark}\label{rem_idea}
In the remainder of this paper we will frequently use the idea above without explicitly referencing it, namely that if two curves $\gamma_1,\gamma_2$ intersect minimally in $\FourPuncturedSphereKh$, then we can count these intersections by looking at intersections of the preimage of $\gamma_2$ and any lift of $\gamma_1$ in $\R^2 \smallsetminus \Z^2$.
\end{remark}
\section{\texorpdfstring{Splitness detection for \(\BNr\) and \(\Khr\)}{Split tangle detection for BNr and Khr}}\label{sec:split_tangle_detection}

\begin{figure}[b]
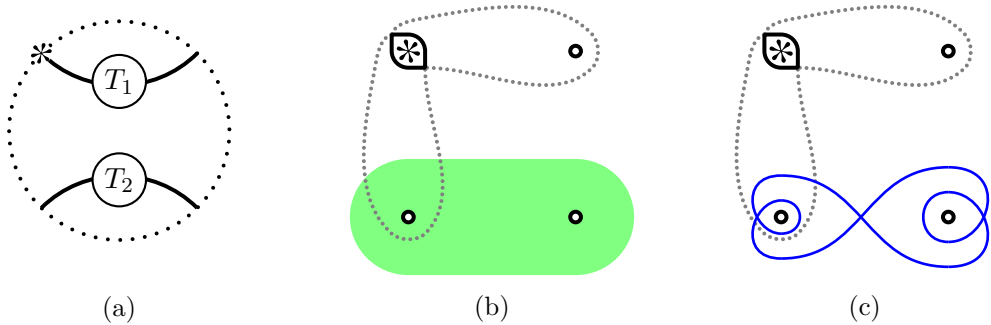

	\centering
	\begin{subfigure}{0.3\textwidth}
		\centering
		\(\splittangle\)
		\caption{}\label{fig:split:tangle}
	\end{subfigure}
	\begin{subfigure}{0.3\textwidth}
		\centering
		\(\slopeZeroNbh\)
		\caption{}\label{fig:split:slope}
	\end{subfigure}
	\begin{subfigure}{0.3\textwidth}
		\centering
		\(\KissingHearts\)
		\caption{}\label{fig:split:kissinghearts}
	\end{subfigure}
	\caption{(a) A split tangle and (b) the region in \(\FourPuncturedSphereKh\) supporting the multicurve invariants of any such tangle, according to Theorem~\ref{thm:split_tangle_detection}. Figure~(c) illustrates the generalized figure-eight curves \(\eight_n(0)\) for \(n=2\); note that \(\eight_1(0)=\rKh_1(0)\). }
	\label{fig:split}
\end{figure} 

Recall that a Conway tangle $T\subset B^3$ is \emph{split} if there exists an essential curve in $\partial B^3 \smallsetminus \partial T$ that bounds a disk in $B^3 \smallsetminus T$. If the slope of this curve is 0, we call the tangle \emph{horizontally split}. Equivalently, a tangle is horizontally split if it can be written as a disjoint union of two-ended tangles \(T_1\) and \(T_2\) as in Figure~\ref{fig:split:tangle}. 
In this section, we show that \(\Khr\) and \(\BNr\) detect this property:

\begin{theorem}
	\label{thm:split_tangle_detection}
	For any Conway tangle \(T\) the following conditions are equivalent: 
	\begin{enumerate}
		\myitem{(1)} \label{enu:split_detection:T} \(T\) is horizontally split;
		\myitem{(2a)} \label{enu:split_detection:shapesBNr}
	Up to some bigrading shift, each component of \(\BNr(T)\) is equal to the horizontal arc
	\(
	\a
	\hateqq
	[\DotB]
	\) 
	or a generalized figure-eight curve
	\(
	\eight_k(0)
	\hateqq
	\Big[
	\begin{tikzcd}[nodes={inner sep=2pt},column sep=14pt]
	\DotB
	\arrow{r}{H^k}
	&
	\DotB
	\end{tikzcd}
	\Big]
	\) 
	for some \(k>0\);
	\myitem{(2b)} \label{enu:split_detection:shapesKhr} Up to some bigrading shift, each component of \(\Khr(T)\) is equal to \(\rKh_1(0)\);
	\myitem{(3a)} \label{enu:split_detection:BNr} Up to homotopy, \(\BNr(T)\) is entirely contained in the shaded region in Figure~\ref{fig:split:slope};
	\myitem{(3b)} \label{enu:split_detection:Khr} Up to homotopy, \(\Khr(T)\) is entirely contained in the shaded region in Figure~\ref{fig:split:slope};
	\myitem{(4a)} \label{enu:split_detection:nodotsBNr}
		The complex \(\DD(T)\hateqq\BNr(T)\) contains no generator $\DotC$;		 
	\myitem{(4b)} \label{enu:split_detection:nodotsKhr}
		The complex \(\DD_1(T)\hateqq\Khr(T)\) contains no generator $\DotC$.
	\end{enumerate}
\end{theorem}

\begin{remark}
	There is an analogous detection result for the Heegaard Floer tangle invariant \(\HFT\) \cite[Theorem~4.1]{LMZ}. 
\end{remark}

By the naturality of \(\BNr\) and \(\Khr\) under twisting \cite[Theorem~1.13]{KWZ}, it follows that a tangle is split if and only if \(\Khr(T)\) consists of rational components of the same slope, or equivalently, if and only if \(\BNr(T)\) consists only of generalized figure-eight curves and arcs of the same slope. 

\begin{proof}
We start with the implication \(\ref{enu:split_detection:T}\Rightarrow\ref{enu:split_detection:shapesBNr}\). 
Bar-Natan associates with the two-ended tangles \(T_1\) and \(T_2\) the invariants \(\KhTl{T_1}\) and \(\KhTl{T_2}\), which are chain complexes over the cobordism category whose objects are crossingless two-ended tangles. Thanks to delooping \cite[Observation~4.18]{KWZ}, we can write these as complexes over the subcategory generated by the trivial tangle~\(\TrivialTwoTangle\). The morphisms in this subcategory can be represented by linear combinations of cobordisms without closed components, ie identity cobordisms with some number of handles attached. 
By Bar-Natan's gluing formalism, \(\KhTl{T}\) is then a tensor product of the complexes \(\KhTl{T_1}\) and \(\KhTl{T_2}\). In particular, the objects of \(\KhTl{T}\) are equal to \(\No\) and the differential consists of linear combinations of identity cobordisms with some number of handles attached to one of the components. Since we are working with coefficients in \(\fieldTwoElements\), the $(4Tu)$-relation \cite[Definition~4.3]{KWZ} allows us to move all handles on these cobordisms to the component containing the basepoint $\ast$ of \(\Lo\). Attaching a handle to the component with a basepoint corresponds to multiplying by $H$ \cite[Definition~4.10]{KWZ}, so \(\KhTl{T}\) is a chain complex over the graded algebra \(\fieldTwoElements[H]\). Therefore, up to homotopy, $\DD(T)$ is a direct sum of complexes of the form
\[
\left[
\begin{tikzcd}[nodes={inner sep=2pt},column sep=14pt]
\DotB
\end{tikzcd}\right]
\qquad\text{or}\qquad
\Big[
\begin{tikzcd}[nodes={inner sep=2pt},column sep=14pt]
\DotB
\arrow{r}{H^k}
&
\DotB
\end{tikzcd}
\Big]
\qquad\text{for some integer \(k>0\),}
\]
as required. 

The implication \(\ref{enu:split_detection:shapesBNr}\Rightarrow\ref{enu:split_detection:shapesKhr}\) follows from the definition of \(\Khr(T)\) as the curve corresponding to the mapping cone \(\DD_1(T)\) of the identity map on \(\DD(T)\) multiplied by \(H\). 
The equivalences $\ref{enu:split_detection:BNr}\Leftrightarrow \ref{enu:split_detection:nodotsBNr}$ and $\ref{enu:split_detection:Khr}\Leftrightarrow \ref{enu:split_detection:nodotsKhr}$ and the implications $\ref{enu:split_detection:shapesBNr} \Rightarrow \ref{enu:split_detection:BNr}$ and $\ref{enu:split_detection:shapesKhr} \Rightarrow \ref{enu:split_detection:Khr}$
are obvious.
The equivalence $\ref{enu:split_detection:nodotsBNr}\Leftrightarrow \ref{enu:split_detection:nodotsKhr}$ follows from the observation that any complex \(\DD(T)\) corresponding to a curve \(\BNr(T)\) contains a generator \(\DotC\) if and only if the same is true for its mapping cone \(\DD_1(T)\). 

The implication \(\ref{enu:split_detection:nodotsBNr}\Rightarrow\ref{enu:split_detection:T}\) remains. This direction relies on a detection result for annular Khovanov homology; this was established by Xie using annular instanton homology. 
We know that the complex $\DD(T)$ representing $\BNr(T)$ only contains generators $\DotB$. 
This is equivalent to saying that the tangle invariant $\KhTl{T}$, as a homotopy equivalence class of chain complexes over $\Cob_{/l}(\Lo\oplus \Li)$, has a representative containing only generators $\Lo$.
Let \(T_a(\infty)\) be the annular link shown in  Figure~\ref{fig:annular}. 
Its annular Khovanov homology $\AKh(T_a(\infty);\F)$ can be computed from $\KhTl{T}$ via gluing arguments similar to~\cite[Section~5]{BarNatanKhT}. It is concentrated in annular grading zero, because in that computation, every circle has winding number zero around the annulus. 
By the universal coefficient theorem, $\AKh(T_a(\infty);\C)$ is also concentrated in annular grading zero. We can now apply Xie's detection result~\cite[Corollary~1.6]{Xie} to deduce that the link $T_a(\infty)$ is contained in a three-ball embedded in the solid torus $S^1\times D^2$.
We conclude with Lemma~\ref{lem:sphere_in_solid_torus} below.
\end{proof}

\begin{figure}[t]
	\centering
	\labellist 
	\pinlabel $T$ at 15 45
	\pinlabel $E$ at 85 45
	\endlabellist
	\includegraphics[width=4.5cm]{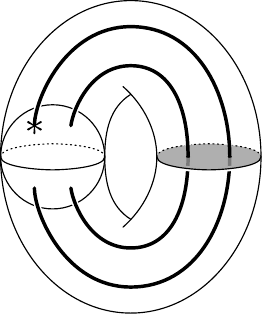}
	\caption{The annular link \(T_a(\infty)\) in $S^1\times D^2$. The shaded disk on the right shows the essential disk \(E\) used in the proof of Lemma~\ref{lem:sphere_in_solid_torus}.}
	\label{fig:annular}
\end{figure} 

\begin{lemma}\label{lem:sphere_in_solid_torus}
	If $T_a(\infty)$ is contained in a three-ball inside a solid torus, then $T$ must be horizontally split.
\end{lemma}

\begin{proof}
	Without loss of generality, we may assume that \(T\) has no unlinked closed components. 

	Let \(S\) be the boundary of the three-ball containing \(T_a(\infty)\) and let \(E\) be the essential disk in $S^1\times D^2$ shown in Figure~\ref{fig:annular}.
	Without loss of generality, we may assume that \(S\) intersects \(E\) transversely, so that \(E\cap S\) is a union of circles. 
	Clearly, \(E\cap S\neq\varnothing\) since \(S\) separates \(E\cap T_a(\infty) \neq\varnothing\) from \(\partial E\). 
	We now consider those circles as subsets of \(S\). One of them, let us call it \(C\), is innermost, so it bounds a disk \(D_S\) in \(S\) disjoint from any other circles. 
	The circle \(C\) also bounds a disk \(D_E\) in \(E\).
	
	Suppose \(D_E\) is disjoint from \(T_a(\infty)\). Then \(D_E\cup D_S\) is disjoint from \(T_a(\infty)\) and bounds a three-ball \(B\) in the solid torus. Since \(T\) has no unlinked closed component, \(B\) is disjoint from \(T_a(\infty)\). Therefore, there exists an isotopy of \(S\) that is the identity outside a small neighbourhood of \(B\) and inside this neighbourhood removes the circle \(C\) from \(E\cap S\neq\varnothing\) as well as any other component of \(E\cap S\neq\varnothing\) that lies in \(D_E\). 
	
	After repeating this procedure a finite number of times, we may assume that \(D_E\) intersects \(T_a(\infty)\) non-trivially. Since the sphere \(D_S\cup D_E\) intersects \(T_a(\infty)\) in an even number of points and \(D_S\) is disjoint from \(T_a(\infty)\), \(D_E\) must contains both intersection points of  \(E\) with \(T_a(\infty)\). So \(E\smallsetminus D_E\) is an annulus and \(D_S\cup (E\smallsetminus D_E)\) is an essential disk which, after isotopy, certifies that \(T\) is a split tangle. 
\end{proof}

\section{Proof of Theorem~\ref{thm:intro:cosmeticrational}}\label{sec:ecsc}

	The two-fold branched cover \(\Sigma(T)\) of $B^3$ branched over $T$ is the exterior of a knot \(K\subset S^3\).
	Suppose, without loss of generality, that \(T(\infty)\) is the unknot. 
 	Then, a curve of slope~\(\infty\) on the boundary of the three-ball containing~\(T\) lifts to a meridian of~\(K\). By adding the appropriate number of twists on the right of the tangle \(T\), we can further assume that \(T\) is  parametrized such that a curve of slope~\(0\) lifts to a longitude of~\(K\). Suppose now that \(T(r)\cong T(r')\) as unoriented links. Then 
	\[
		S^3_r(K)= \Sigma(T(r))\cong \Sigma(T(r'))=S^3_{r'}(K).
	\]
	According to \cite[Theorem~2]{HanselmanCSC}, this implies that \(r=-r'\) where
	\[
	r=\pm2 
	\text{ (case 1)} \quad \text{or} \quad
	r=\nicefrac{\pm1}{n} 
	\text{ for some positive integer }n \text{ (case 2)}
	\]
	Moreover, as a consequence of our chosen parametrization, the connectivity of the tangle~\(T\) is \(\No\). 
	This can be seen as follows: First, observe that \(T\) has no closed component and the connectivity of the tangle is not \(\Ni\). Both follow from the fact that \(T(\infty)\) is the unknot.  
	Secondly, if \(T(0)\) were a knot, its two-fold branched cover would be a rational homology sphere. 
	This contradicts the fact that the two-fold branched cover is 0-surgery on the knot \(K\subset S^3\). 
	So \(T(0)\) is a two-component link and hence the connectivity of \(T\) is not \(\Nio\). 
	
	The strategy for the proof is to compare the reduced Khovanov homologies \(\Khr(T(r))\) and \(\Khr(T(r'))\). 
	We first equip \(T(r)\) and \(T(r')\) with orientations such that they agree as oriented links. 
	Then \(\Khr(T(r))\) and \(\Khr(T(r'))\) agree as absolutely bigraded groups. 
	We work with coefficients in~\(\F\), so that reduced Khovanov homology is independent of the reduction point \cite[Corollary~3.2.C]{Shumakovitch}. 
	We will compute \(\Khr(T(r))\) and \(\Khr(T(r'))\) by pairing the $\BNr$-invariants of the rational tangle fillings (arcs) with the multicurve \(\textcolor{blue}{C}\coloneqq\Khr(T)\).
	A priori, the absolute bigrading on \(\textcolor{blue}{C}\) depends on the orientation of the tangle, but as we will see below, it is in fact orientation independent. 
	Since \(T(\infty)\) is the unknot, we know that \(\textcolor{blue}{C}\) has only one intersection with the vertical arc \(\a_\infty\coloneqq\BNr(\Li)\). 
	Special curves $\sKh_{2n}(s)$ intersect \(\a_\infty\) in more than one point, unless they have slope \(s=\infty\), in which case they are disjoint from \(\a_\infty\). 
	Similarly, a rational curve $\rKh_{n}(s)$ intersects \(\a_\infty\) in more than one point, unless $n=1$ and \(s\in \mathbb Z\), in which case there is a single intersection point. 
	Hence, we may write
	\[
	\textcolor{blue}{C}
	\coloneqq 
	\textcolor{blue}{\gamma_1}
	\cup \dots\cup
	\textcolor{blue}{\gamma_m}
	\cup
	\textcolor{blue}{\rho},
	\]
	where 
	\(
	\textcolor{blue}{\gamma_1},
	\dots,
	\textcolor{blue}{\gamma_m}
	\) 
	are special components of slope \(\infty\) and \(\textcolor{blue}{\rho}=\rKh_1(s)\), of slope \(s\in\mathbb Z\). Since \(T\) is non-rational, \(m>0\) by Theorem~\ref{thm:rational_tangle_detection}.
	Moreover, by Theorem~\ref{thm:detection:connectivity}, we know that \(s\) is an even integer. 
	
	We now consider the two cases separately. The arguments in both cases are essentially the same. We first show that the slope \(s\) of \(\textcolor{blue}{\rho}\) must be 0 for the total dimensions of \(\Khr(T(r))\) and \(\Khr(T(r'))\) to agree; then we compute the absolute quantum gradings and observe that they are different.

	\medskip\noindent {\bf Case 1: \(\{r,r'\}=\{\pm2\}\). } 
	Since the connectivity of the tangle \(T\) is \(\No\), \(T(+2)=T(-2)\) is a link with two components. 
	Consider their linking number. If we choose the same orientation of the tangle \(T\), the linking numbers of \(T(+2)\) and \(T(-2)\) are different, since the crossings in the \(\pm2\)-twist tangles then have different signs. 
	So up to an overall orientation reversal (which does not affect the reduced Khovanov homology), we may assume that the orientations on \(T(+2)\) and \(T(-2)\) are as follows:
	\[
	T(+2)=
	\Tpii
	\qquad
	T(-2)=
	\Tmii
	\]
	Since the crossings in the \(\pm2\)-twist tangles are all positive, the linking number of the tangle \(T\) with the orientation as in \(T(+2)\) is the same as with the orientation as in \(T(-2)\). (The linking number of a tangle is defined in \cite[Definition~4.7]{KWZ}.)  Since the two orientations are obtained by reversing one strand, these linking numbers also differ by a sign, so the linking number of $T$ is zero. Hence \(\textcolor{blue}{C}=\Khr(\textcolor{blue}{T})\) is independent of orientations; see for example~\cite[Proposition~4.8]{KWZ}.
	 
	Define two arcs
	\[
	\textcolor{red}{\a_+}\hateqq
	\left[
	\begin{tikzcd}[nodes={inner sep=2pt},column sep=14pt, red]
	\GGzqh{\DotCred}{}{-5}{}
	\arrow{r}{D}
	&
	\GGzqz{\DotCred}{}{-3}{}
	\arrow{r}{S}
	&
	\GGzqz{\DotBred}{}{-2}{}
	\end{tikzcd}\right]
	\qquad
	\text{and}
	\qquad
	\textcolor{red}{\a_-}\hateqq
	\left[
	\begin{tikzcd}[nodes={inner sep=2pt},column sep=14pt, red]
	\GGzqh{\DotBred}{}{-4}{}
	\arrow{r}{S}
	&
	\GGzqz{\DotCred}{}{-3}{}
	\arrow{r}{D}
	&
	\GGzqz{\DotCred}{}{-1}{}
	\end{tikzcd}\right]
	\]
	These are the arc invariants of the mirrors of the \(\pm2\)-twist tangles in \(T(+2)\) and \(T(-2)\), respectively. (For instance, this calculation follows from \cite[Example~4.27]{KWZ}, using (1) the relation $q=2(h+\delta)$ between the $\delta$-, homological, and quantum gradings in Khovanov theory, and (2) the formula from \cite[Proposition~4.8]{KWZ} for the grading shift induced by reversing the orientation of a tangle component.) 
	Then, by the pairing theorem,
	\[
	\Khr\Big(T(\pm2)\Big)
	\cong
	\HF(\textcolor{red}{\a_\pm},\textcolor{blue}{C})
	=
	\HF(\textcolor{red}{\a_\pm},\textcolor{blue}{\gamma_1})
	\oplus
	\dots
	\oplus
	\HF(\textcolor{red}{\a_\pm},\textcolor{blue}{\gamma_m})
	\oplus
	\HF(\textcolor{red}{\a_\pm},\textcolor{blue}{\rho})
	\]
	The total dimensions of the first \(m\) pairs of summands are identical. 
	By Lemma~\ref{lem:pairing_linear_curves:dimension_formula}, and  because a figure eight and an arc of the same slope intersect minimally in two points, the dimensions of the final summands are 
	\[\HF(\textcolor{red}{\a_\pm},\textcolor{blue}{\rho}) = 
	\begin{cases}
	3\mp1,& \text{ if } s=2; \\
	3\pm1,& \text{ if } s=-2; \\
	|s\mp 2|,& \text{ otherwise.}
	\end{cases}\]
	Our assumption that \(T(+2)\cong T(-2)\) implies that the dimensions of reduced Khovanov homology agree, so the slope \(s=0\). 
	
	We now consider the quantum gradings. 
	Recall that the grading on $\Khr(\Lo)$ is independent of the orientation on $\Lo$. 
	Since $T(\infty)$ is the unknot, the quantum gradings on $\textcolor{blue}{\rho}$ and $\Khr(\Lo)\hateqq\Big[\GGzqz{\DotB}{}{-1}{}\xrightarrow{H}\GGzqz{\DotB}{}{+1}{}\Big]$ agree.  
	Thus, \(\HF(\textcolor{red}{\a_+},\textcolor{blue}{\rho})\) and \(\HF(\textcolor{red}{\a_-},\textcolor{blue}{\rho})\) are graded isomorphic, because both are graded isomorphic to the reduced Khovanov homology of the same oriented Hopf link. 
	Moreover, for all \(i=1,\dots,m\), the quantum grading is shifted such that
	\[
	\HF(\textcolor{red}{\a_-},\textcolor{blue}{\gamma_i})
	\cong
	q^{+2}\HF(\textcolor{red}{\a_+},\textcolor{blue}{\gamma_i})
	\]
	which together with the previous observation contradicts \(\Khr(T(+2))\cong\Khr(T(-2))\).
	The grading shift for \(\HF(\textcolor{red}{\a_\pm},\textcolor{blue}{\gamma_i})\) can be seen as follows: After pulling the curves \(\textcolor{blue}{\gamma_i}\) sufficiently tight, their intersections with the arcs \(\textcolor{red}{\a_\pm}\) are all in a small neighbourhood of the intersection points of \(\textcolor{red}{\a_\pm}\) with the parametrizing arcs corresponding to the generators \(\DotBred\) as shown in Figure~\ref{fig:ecsc:case_i:specials} from the viewpoint of the planar cover of $\FourPuncturedSphereKh$. 
	The arcs \(\textcolor{red}{\a_\pm}\) are parallel in this region, so the grading difference is precisely the (negative of) the grading difference between these two generators. 
	
	\begin{figure}[b]
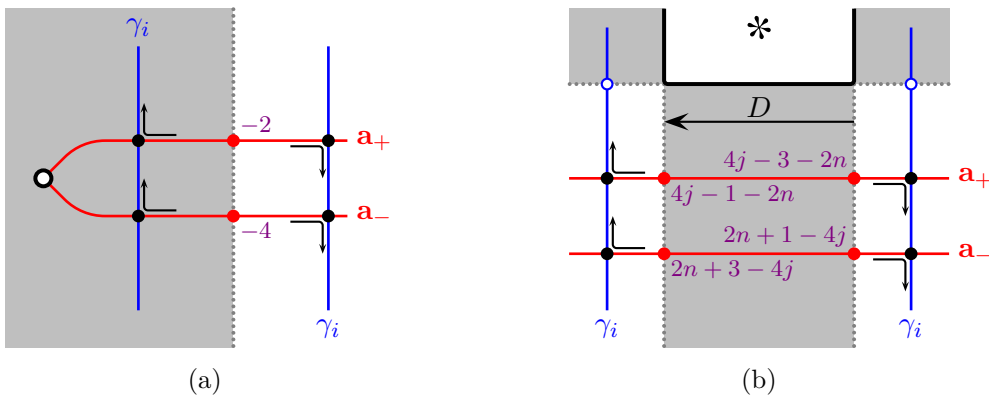

		\centering
		\begin{subfigure}[b]{0.45\textwidth}
			\centering
			\(\ECSCIspecials\)
			\caption{}\label{fig:ecsc:case_i:specials}
		\end{subfigure}
		\begin{subfigure}[b]{0.45\textwidth}
			\centering
			\(\ECSCIIspecials\)
			\caption{}\label{fig:ecsc:case_ii:specials}
		\end{subfigure}
		\caption{Canonical generators of \(\HF(\textcolor{red}{\a_\pm},\textcolor{blue}{\gamma_i})\) in the proof of Theorem~\ref{thm:intro:cosmeticrational}, (a)~Case~1 and (b)~Case~2. The arrows indicate the corresponding morphisms used for computing the gradings, as in Section~\ref{subsec:background:bigrading}. }
		\label{fig:ecsc}
	\end{figure} 
	
	\medskip\noindent {\bf Case 2: \(\{r,r'\}=\{\nicefrac{\pm1}{n}\}\). } 
	Since the connectivity of the tangle \(T\) is \(\No\), \(T(\nicefrac{1}{n})=T(\nicefrac{-1}{n})\) is a knot and up to overall orientation reversal, the orientations on \(T(\nicefrac{1}{n})\) and \(T(\nicefrac{-1}{n})\) are as follows:
	\[
	T(\nicefrac{1}{n})=
	\Tpn
	\qquad
	T(\nicefrac{-1}{n})=
	\Tmn
	\]
	Define
	\begin{align*}
	\textcolor{red}{\a_+}
	&\hateqq
	\left[
	\begin{tikzcd}[nodes={inner sep=2pt},column sep=14pt,ampersand replacement=\&, red]
	\GGzqh{\DotCred}{}{-2n}{}
	\arrow{r}{S}
	\&
	\GGzqh{\DotBred}{}{1-2n}{}
	\arrow{r}{D}
	\&
	\GGzqz{\DotBred}{}{3-2n}{}
	\arrow[dotted]{r}{S^2}
	\&
	\quad
	\arrow[dotted]{r}{S^2}
	\&
	\GGzqz{\DotBred}{}{4j-3-2n}{}
	\arrow{r}{D}
	\&
	\GGzqz{\DotBred}{}{4j-1-2n}{}
	\arrow[dotted]{r}{S^2}
	\&
	\quad
	\arrow[dotted]{r}{}
	\&
	\GGzqz{\DotBred}{}{-1}{}
	\end{tikzcd}\right]
	\text{ and}
	\\
	\textcolor{red}{\a_-}
	&\hateqq
	\left[
	\begin{tikzcd}[nodes={inner sep=2pt},column sep=14pt,ampersand replacement=\&, red]
	\GGzqz{\DotBred}{}{1}{}
	\arrow[dotted]{r}{}
	\&
	\quad
	\arrow[dotted]{r}{S^2}
	\&
	\GGzqz{\DotBred}{}{2n+1-4j}{}
	\arrow{r}{D}
	\&
	\GGzqz{\DotBred}{}{2n+3-4j}{}
	\arrow[dotted]{r}{S^2}
	\&
	\quad
	\arrow[dotted]{r}{S^2}
	\&
	\GGzqz{\DotBred}{}{2n-3}{}
	\arrow{r}{D}
	\&
	\GGzqz{\DotBred}{}{2n-1}{}
	\arrow{r}{S}
	\&
	\GGzqz{\DotCred}{}{2n}{}
	\end{tikzcd}\right]
	\end{align*}
	Note that \(j=1,\dots,\lfloor\tfrac{n}{2}\rfloor\). 
	The arcs \(\textcolor{red}{\a_+}\) and \(\textcolor{red}{\a_-}\) are the  \(\BNr\) invariants of the mirrors of the \(\nicefrac{\pm1}{n}\)-twist tangles in \(T(\nicefrac{1}{n})\) and \(T(\nicefrac{-1}{n})\), respectively; see \cite[Example~6.2, Proposition~4.8]{KWZ}. So as in Case~1, the pairing theorem allows us to write
	\[
	\Khr\Big(T(\nicefrac{\pm1}{n})\Big)
	\cong
	\HF(\textcolor{red}{\a_\pm},\textcolor{blue}{C})
	=
	\HF(\textcolor{red}{\a_\pm},\textcolor{blue}{\gamma_1})
	\oplus
	\dots
	\oplus
	\HF(\textcolor{red}{\a_\pm},\textcolor{blue}{\gamma_m})
	\oplus
	\HF(\textcolor{red}{\a_\pm},\textcolor{blue}{\rho})
	\]
	The total dimensions of the first \(m\) pairs of summands are identical, regardless of the slope \(s\) of the rational component \(\textcolor{blue}{\rho}\). 
	Since the slope \(s\) is an even integer, it never agrees with \(\nicefrac{\pm1}{n}\), so the dimensions of the final summands \(\HF(\textcolor{red}{\a_\pm},\textcolor{blue}{\rho})\) are equal to \(|1\mp sn|\) by Lemma~\ref{lem:pairing_linear_curves:dimension_formula} and hence only agree if \(s=0\).
	
	We now compute quantum gradings. First,  \(\HF(\textcolor{red}{\a_\pm},\textcolor{blue}{\rho})\) agree as absolutely bigraded homology groups, since they compute the reduced Khovanov homology of an unknot, shifted in quantum grading by the same amount. To compute the grading shifts for the first \(m\) summands, we observe that after pulling the curves \(\textcolor{blue}{\gamma_i}\) sufficiently tight (see \cite[Definition~6.1]{KWZthinness} and the discussion afterwards),  the intersection points between the arcs \(\textcolor{red}{\a_\pm}\) and \(\textcolor{blue}{\gamma_i}\) sit in a small neighbourhood of the vertical line through the special marked point. 
	If \(n\) is even, the relevant portions of the complexes \(\textcolor{red}{\a_+}\) and \(\textcolor{red}{\a_-}\) are, respectively,  
	\[
	\left(
	\begin{tikzcd}[nodes={inner sep=2pt},column sep=14pt,ampersand replacement=\&, red]
	\GGzqz{\DotBred}{}{4j-3-2n}{}
	\arrow{r}{D}
	\&
	\GGzqz{\DotBred}{}{4j-1-2n}{}
	\end{tikzcd}
	\right)
	\qquad\text{and}\qquad
	\left(
	\begin{tikzcd}[nodes={inner sep=2pt},column sep=14pt,ampersand replacement=\&, red]
	\GGzqz{\DotBred}{}{2n+1-4j}{}
	\arrow{r}{D}
	\&
	\GGzqz{\DotBred}{}{2n+3-4j}{}
	\end{tikzcd}
	\right)
	\]
	where \(j=1,\dots,\lfloor\tfrac{n}{2}\rfloor\). The corresponding curve segments are illustrated in Figure~\ref{fig:ecsc:case_ii:specials}. 
	They are obviously parallel, so there is a one-to-one correspondence between generators \(x_+\in \HF(\textcolor{red}{\a_+},\textcolor{blue}{\gamma_i})\) and generators \(x_-\in \HF(\textcolor{red}{\a_-},\textcolor{blue}{\gamma_i})\)
	such that the quantum gradings satisfy
	\[
	q(x_-)-q(x_+)=8j-4-4n<0
	\] 
	for \(j=1,\dots,\lfloor\tfrac{n}{2}\rfloor\). 
	If \(n\) is odd, there are additional generators stemming from the generators 	
	\(\GGzqz{\DotBred}{}{-1}{}\) and \(\GGzqz{\DotBred}{}{+1}{}\) of the complexes \(\textcolor{red}{\a_\pm}\). 
	The corresponding curve segments look as in Figure~\ref{fig:ecsc:case_i:specials}, except that the quantum gradings are different. The correspondence from the case that \(n\) is even extends to the case that \(n\) is odd so that the quantum gradings of the additional generators satisfy
	\[
	q(x_-)-q(x_+)=-2<0
	\]
	The grading shifts are strictly negative in all cases, contradicting \(\Khr(T(\nicefrac{1}{n}))\cong\Khr(T(\nicefrac{-1}{n}))\).\qed

\begin{remark}


This proof highlights the utility of the quantum gradings in Khovanov homology.  However, it also suggests an alternate strategy that avoids gradings in these complexes altogether through the exact triangle.  In general, by using the immersed curves reformulation of Khovanov homology, one is able to split up the skein exact triangle into several summands. There is one for each component of the immersed multicurve for the tangle complementary to the crossing where the exact triangle is being implemented. For each exact triangle, the dimensions of the three groups are simply computed by a count of intersections between the component and three rational curves in the four-punctured sphere of distance one. This gives much stronger constraints on the structure of the exact triangle. Frequently, as a result, the maps in the exact triangle can be computed as well, and additional grading structures can be deduced.  A similar perspective on these exact sequences is seen in bordered Heegaard Floer homology for three-manifolds with torus boundary \cite[Section 11.2]{LOT}.
\end{remark}

\section{The Generalized Cosmetic Crossing Conjecture holds asymptotically}\label{sec:agccc}

Throughout this section we fix a Conway tangle $T$ with connectivity $\No$ and without any closed components. Furthermore, we consider the family of knots \(\{K_n\}_{n\in\Z}\) shown in Figure~\ref{fig:Tpmiin} and defined by  $K_n=T(\nicefrac{1}{2n})$ for \(n\in\Z\).   
Equivalently, each knot \(K_n\) is the result of a band surgery on a fixed two-component link $T(0)$ and the knots \(K_n\) and \(K_{n+1}\) are obtained from each other by adding a full twist to the band. 
This point of view explains the restrictions placed on the tangle~\(T\).

\begin{figure}[b]
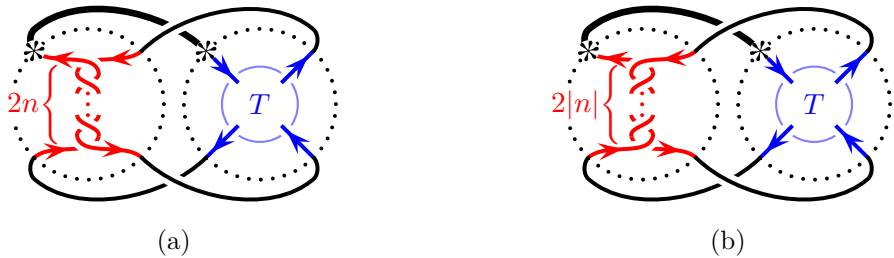

	\centering
	\begin{subfigure}[b]{0.45\textwidth}
		\centering
		\(\Tpiin\)
		\caption{}\label{fig:Tpiin}
	\end{subfigure}
	\begin{subfigure}[b]{0.45\textwidth}
		\centering
		\(\Tmiin\)
		\caption{}\label{fig:Tmiin}
	\end{subfigure}
	\caption{The knots $K_n$ for \(n\geq0\) (a) and \(n\leq0\) (b)}
	\label{fig:Tpmiin}
\end{figure} 

If $T$ is horizontally split, then all $K_n$ are equal to each other. We now restate the two conjectures on cosmetic crossings in terms of Conway tangles:

\begin{conjecture}[Cosmetic Crossing Conjecture]
	Suppose \(T\) is not horizontally split. 
	Then \(K_0\) and \(K_1\) are different as unoriented knots.
\end{conjecture}

\begin{conjecture}[Generalized Cosmetic Crossing Conjecture]
	Suppose \(T\) is not horizontally split. 
	Then the unoriented knots \(\{K_n\}_{n\in \Z}\) are pairwise different.
\end{conjecture}

Recently, Wang showed that these conjectures hold assuming $T(0)$ is a split link \cite{Wang}. While he states these conjectures for oriented knots, we are not aware of any counterexamples to the conjectures as stated above for unoriented knots. Note, however, that a crossing change may result in the mirror of the original knot, as illustrated by the \((3,-3,\pm1)\)-pretzel knots, see the remarks to \cite[Problem~1.58]{Kirby}. 

Below we restate and prove Theorem~\ref{thm:agccc}, showing that the Generalized Cosmetic Crossing Conjecture holds ``asymptotically'', that is, for $n$ large enough:
\begin{theorem}\label{thm:main}
	Suppose \(T\) is not horizontally split. 
	Then there exists an integer \(N\) such that the knots \(\{K_n\}_{|n|\geq N}\) are pairwise different as unoriented knots. 
\end{theorem}

\begin{lemma}\label{lem:only_zero_slopes_implies_GCCC}
	Suppose \(\Khr(T)\) only contains curves of slope 0 of which at least one is special. Then \(K_n\not\cong K_m\) for any \(n\neq m\). 
\end{lemma}

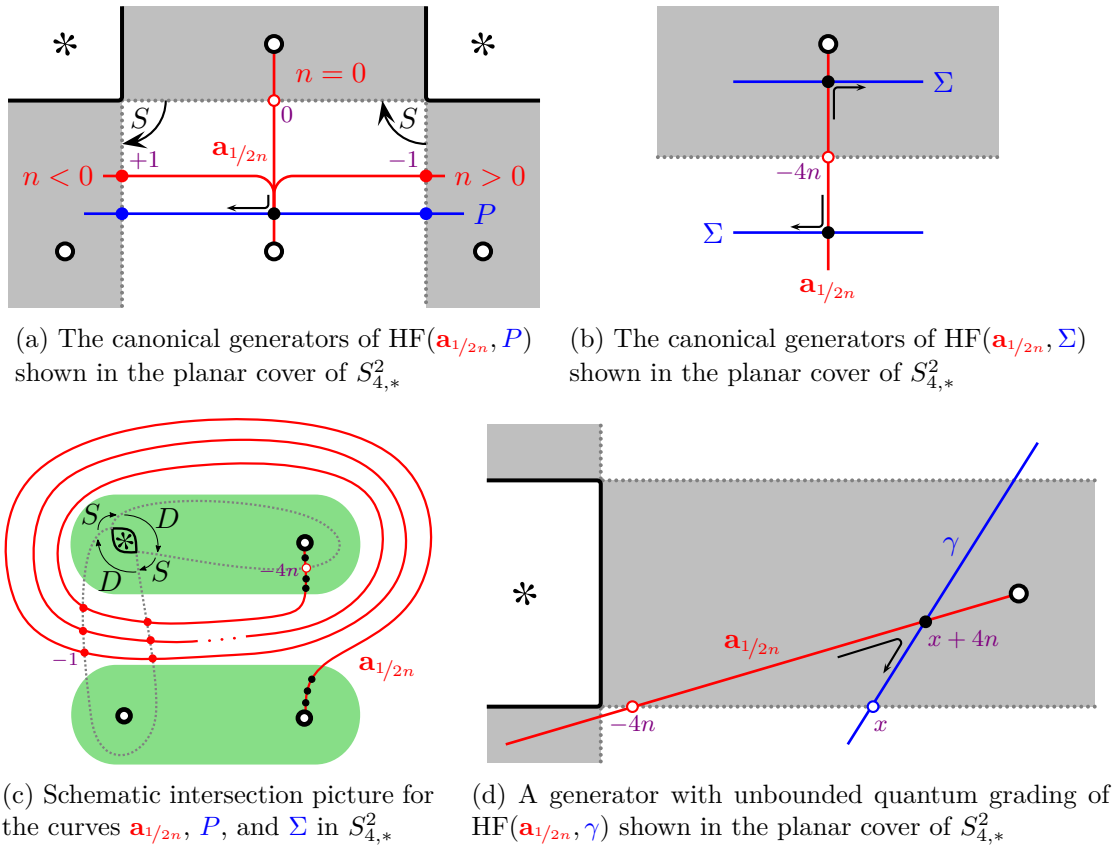
\begin{figure}[b]
	\centering
	\begin{subfigure}[b]{0.45\textwidth}
		\centering
		\(\AGCCCzeroRationals\)
		\caption{The canonical generators of
		\(\HF(\textcolor{red}{\a_{\nicefrac{1}{2n}}},\textcolor{blue}{P})\) shown in the planar cover of $\FourPuncturedSphereKh$}\label{fig:agccc:zero:rationals}
	\end{subfigure}
	\begin{subfigure}[b]{0.45\textwidth}
		\centering
		\(\AGCCCzeroSpecials\)
		\caption{The canonical generators of
		\(\HF(\textcolor{red}{\a_{\nicefrac{1}{2n}}},\textcolor{blue}{\Sigma})\) shown in the planar cover of $\FourPuncturedSphereKh$}\label{fig:agccc:zero:specials}
	\end{subfigure}
	\bigskip\\
	\begin{subfigure}[b]{0.38\textwidth}
		\centering
		\centering
		\labellist 
		\pinlabel $S$ at 27 79
		\pinlabel $D$ at 51 78
		\pinlabel $S$ at 49 61
		\pinlabel $D$ at 33 58
		\pinlabel $\textcolor{red}{\a_{\nicefrac{1}{2n}}}$ at 120 30
		\pinlabel \tiny $\textcolor{violet}{-4n}$ at 86 61
		\pinlabel \tiny $\textcolor{violet}{-1}$ at 19 33
		\endlabellist
		\includegraphics[scale=1.2]{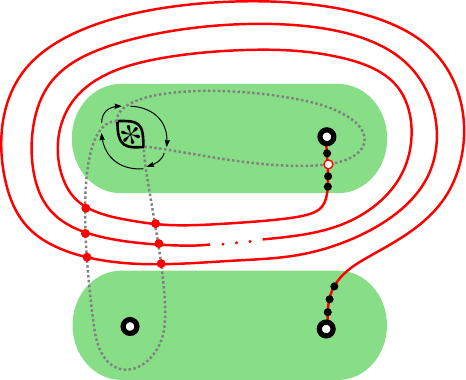}
		\caption{Schematic intersection picture for the curves $\textcolor{red}{\a_{\nicefrac{1}{2n}}}$,  $\textcolor{blue}{P}$, and $\textcolor{blue}{\Sigma}$ in $\FourPuncturedSphereKh$
	}\label{fig:agccc:zero:downstairs}
	\end{subfigure}
	\begin{subfigure}[b]{0.55\textwidth}
		\centering
		\(\AGCCCnoupperbound\)
		\caption{A generator with unbounded quantum grading of
			\(\HF(\textcolor{red}{\a_{\nicefrac{1}{2n}}},\textcolor{blue}{\gamma})\) shown in the planar cover of $\FourPuncturedSphereKh$}\label{fig:agccc:zero:noupperbound}
	\end{subfigure}
	\caption{Illustration of intersection points between various curves in the proofs of Lemmas~\ref{lem:only_zero_slopes_implies_GCCC} and~\ref{lem:nonzero_slope_implies_limit_iinfinitty}. Numbers near generators (intersection points) indicate their quantum gradings. 
	In (c), the dots $\bullet$ indicate the intersection points of the curve $\textcolor{red}{\a_{\nicefrac{1}{2n}}}$  with $\textcolor{blue}{P}$ (contained in the bottom region) and $\textcolor{blue}{\Sigma}$ (contained in the top region).
}
	\label{fig:agccc:zero}
\end{figure}

\begin{proof}
	Let us write \(\Khr(T)=\textcolor{blue}{\Sigma}\cup \textcolor{blue}{P}\), where \(\textcolor{blue}{\Sigma}\) consists of special components and \(\textcolor{blue}{P}\) of rational components. By assumption \(\textcolor{blue}{\Sigma}\neq\varnothing\). Also \(\textcolor{blue}{P}\neq\varnothing\), since the pairing of
	the arc \(\BNr(\Lo)\) with \(\Khr(T)\) computes the Khovanov homology of the link $T(0)$ and as such is non-zero. 
	Since the connectivity of the tangle \(T\) is \(\No\), the induced orientation of \(K_n\) on the tangle \(T\) and its rational filling are as shown in Figure~\ref{fig:Tpiin}, up to overall orientation reversal. 
	Define the arc of slope $\nicefrac{1}{2n}$:
	\[
	\textcolor{red}{\a_{\nicefrac{1}{2n}}}
	\hateqq
	\begin{cases*}
	\left[
	\begin{tikzcd}[nodes={inner sep=2pt},column sep=14pt,ampersand replacement=\&, red]
	\GGzqh{\DotCred}{}{-4n}{}
	\arrow{r}{S}
	\&
	\GGzqh{\DotBred}{}{1-4n}{}
	\arrow{r}{D}
	\&
	\GGzqz{\DotBred}{}{3-4n}{}
	\arrow[dotted]{r}{S^2}
	\&
	\quad
	\arrow[dotted]{r}{}
	\&
	\GGzqh{\DotBred}{}{-3}{}
	\arrow{r}{D}
	\&
	\GGzqz{\DotBred}{}{-1}{}
	\end{tikzcd}\right]
	& for \(n>0\)
	\\
	\left[
	\begin{tikzcd}[nodes={inner sep=2pt},column sep=14pt,ampersand replacement=\&, red]
	\GGzqz{\DotCred}{}{0}{}
	\end{tikzcd}\right]
	& for \(n=0\)
	\\
	\left[
	\begin{tikzcd}[nodes={inner sep=2pt},column sep=14pt,ampersand replacement=\&, red]
	\GGzqz{\DotBred}{}{1}{}
	\arrow{r}{D}
	\&
	\GGzqz{\DotBred}{}{3}{}
	\arrow[dotted]{r}{}
	\&
	\quad
	\arrow[dotted]{r}{S^2}
	\&
	\GGzqz{\DotBred}{}{-4n-3}{}
	\arrow{r}{D}
	\&
	\GGzqz{\DotBred}{}{-4n-1}{}
	\arrow{r}{S}
	\&
	\GGzqz{\DotCred}{}{-4n}{}
	\end{tikzcd}
	\right]
	& for \(n<0\)
	\end{cases*}
	\]
	By the same calculation as in the proof of Theorem~\ref{thm:intro:cosmeticrational}, Case 2 (using \cite[Example~6.2, Proposition~4.8]{KWZ}), these are the arc invariants of the mirrors of the rational fillings of \(T\).
	Then, by the pairing theorem, 
	\[
	\Khr(K_n)
	\cong
	\HF(\textcolor{red}{\a_{\nicefrac{1}{2n}}},\Khr(T))
	=
	\HF(\textcolor{red}{\a_{\nicefrac{1}{2n}}},\textcolor{blue}{\Sigma})
	\oplus
	\HF(\textcolor{red}{\a_{\nicefrac{1}{2n}}},\textcolor{blue}{P})
	\]
	We now study how both summands behave when varying \(n\);  Figure~\ref{fig:agccc:zero:downstairs} depicts the schematic picture of intersections in $\FourPuncturedSphereKh$. After pulling the multicurves \(\textcolor{blue}{P}\) and \(\textcolor{red}{\a_{\nicefrac{1}{2n}}}\) sufficiently tight, the intersection points generating \(\HF(\textcolor{red}{\a_{\nicefrac{1}{2n}}},\textcolor{blue}{P})\) all sit close to the end of \(\textcolor{red}{\a_{\nicefrac{1}{2n}}}\) corresponding to the generator \(\GGzqz{\DotBred}{}{-1}{0}\) for \(n>0\), \(\GGzqz{\DotCred}{}{0}{0}\) for \(n=0\), and \(\GGzqz{\DotBred}{}{+1}{0}\) for \(n<0\). 
	The corresponding curve segments are shown in Figure~\ref{fig:agccc:zero:rationals}. We see that the quantum grading of \(\HF(\textcolor{red}{\a_{\nicefrac{1}{2n}}},\textcolor{blue}{P})\) is independent of \(n\); see Section~\ref{subsec:background:bigrading} for how gradings are computed. 
	
	We now investigate the pairing of $\textcolor{red}{\a_{\nicefrac{1}{2n}}}$ with special curves. After pulling all multicurves sufficiently tight, the intersection points generating \(\HF(\textcolor{red}{\a_{\nicefrac{1}{2n}}},\textcolor{blue}{\Sigma})\) all sit close to the end of \(\textcolor{red}{\a_{\nicefrac{1}{2n}}}\) corresponding to the generator \(\GGzqz{\DotCred}{}{-4n}{0}\), see Figure~\ref{fig:agccc:zero:specials}. Thus, the shift in quantum grading is as follows: 
	\[
	\HF(\textcolor{red}{\a_{\nicefrac{1}{2n}}},\textcolor{blue}{\Sigma})
	\cong 
	q^{4n}
	\HF(\textcolor{red}{\a_{\infty}},\textcolor{blue}{\Sigma})
	\cong 
	q^{4n-4m}
	\HF(\textcolor{red}{\a_{\nicefrac{1}{2m}}},\textcolor{blue}{\Sigma})
	\]
	Therefore \(K_n\not\cong K_m\) if \(n\neq m\). 
\end{proof}

\begin{lemma}\label{lem:nonzero_slope_implies_limit_iinfinitty}
	Suppose \(\Khr(T)\) contains a curve of a non-zero slope. Then there exists $N$ such that the unoriented knots \(\{K_n\}_{|n|\geq N}\) are all different. 
\end{lemma}

\begin{proof}
	Let us write 	\(\Khr(T)=\textcolor{blue}{\gamma_1}\cup \cdots\cup\textcolor{blue}{\gamma_m}\)	for some integer \(m>0\). 
	For each \(i=1,\dots,m\), let \((p_i,q_i)\) be a pair of mutually prime integers such that the slope of \(\textcolor{blue}{\gamma_i}\) is \(\nicefrac{p_i}{q_i}\). By assumption, there is some \(i\in\{1,\dots,m\}\) such that \(p_i\neq0\). 
	Let 
	\[
	M=\max\left\{\left.\tfrac{|q_i|}{2|p_i|}\right| i\in\{1,\dots,m\}\co p_i\neq0\right\}
	\]
	Then for \(|n|> M\), the slope \(\nicefrac{1}{2n}\) of the curve \(\textcolor{red}{\a_{\nicefrac{1}{2n}}}\) is distinct from \(\nicefrac{p_i}{q_i}\). Therefore, by Lemma~\ref{lem:pairing_linear_curves:dimension_formula}, 
	\[
	\dim\HF(\textcolor{red}{\a_{\nicefrac{1}{2n}}},\textcolor{blue}{\gamma_i})
	=
	\ell_i\cdot
	|q_i-p_i\cdot 2n|
	\]
	where \(\ell_i\) is the length of \(\textcolor{blue}{\gamma_i}\). As we have seen in the proof of Lemma~\ref{lem:only_zero_slopes_implies_GCCC}, if \(p_i=0\), the dimension of \(\HF(\textcolor{red}{\a_{\nicefrac{1}{2n}}},\textcolor{blue}{\gamma_i})\) is independent of \(n\). If \(p_i\neq0\), the sign of the expression 
	\(\nicefrac{q_i}{p_i}-2n\)	is the same for all \(n>M\). 
	Therefore, \(\dim\Khr(K_n)\) is a strictly increasing function in \(n\) for $n>M$. The same argument shows that it is strictly decreasing for $n<-M$. 
	Thus the knots \(\{K_n\}_{n>M}\) are pairwise different, and so are the knots \(\{K_n\}_{n<-M}\).

	It remains to distinguish the two families. For this we first prove that the quantum grading of $\{\Khr(K_n)\}_{n\gg0}$ is unbounded above and bounded below. 
	
	The existence of a lower bound follows from the following two observations: First, the quantum gradings of the generators of $\textcolor{red}{\a_{\nicefrac{1}{2n}}}$ are bounded above by $-1$. Second, every intersection point generating the Lagrangian Floer homology between $\textcolor{red}{\a_{\nicefrac{1}{2n}}}$ and a rational or special curve can be represented by a homogeneous morphism containing a component labelled by an algebra element of quantum grading greater than or equal to $-2$ (namely one of the algebra elements \(\id,S,S^2, D\in\BNAlgH\)); this follows from an elementary argument about straight lines in the covering space $\FourPuncturedSphereKh$. Thus, if the minimal grading of a generator of $\DD_1(T)\hateqq\textcolor{blue}{\Khr(T)}$ is $\mu$, the formula from Section~\ref{subsec:background:bigrading} for computing the quantum grading of generators of $\HF(\textcolor{red}{\a_{\nicefrac{1}{2n}}},\textcolor{blue}{\Khr(T)})$ gives us $\mu-(-1)+(-2)$ as a lower bound.  
	
	Next, we show that the quantum grading of $\{\Khr(K_n)\}_{n\gg0}$ has no upper bound. 
	By assumption, there exists a component $\textcolor{blue}{\gamma}$ of $\textcolor{blue}{\Khr(T)}$ of non-zero slope. For $n\gg0$, we may assume that the slope of $\textcolor{blue}{\gamma}$ is bigger than the slope $\nicefrac{1}{2n}$ of the arc $\textcolor{red}{\a_{\nicefrac{1}{2n}}}$. Then, there exists an intersection point close to the generator $\GGzqz{\DotCred}{}{-4n}{}$, which looks like Figure~\ref{fig:agccc:zero:noupperbound}. Clearly, the quantum grading of this generator is unbounded. 

	Analogous arguments imply that the quantum grading of $\{\Khr(K_n)\}_{n\ll0}$ is unbounded below and bounded above. This proves that there exists $N\gg M$ such that knots in \(\{K_n\}_{n>N} \cup \{K_n\}_{n<-N}\) are pairwise different.
\end{proof}

\begin{proof}[Proof of Theorem~\ref{thm:main}]
We will show that the reduced Khovanov homology of the knots \(\{K_n\}_{|n|\geq N}\) are pairwise different by studying how the invariant \(\Khr(T)\) pairs with the arc \(\textcolor{red}{\a_{\nicefrac{1}{2n}}}\coloneqq\BNr(Q_{\nicefrac{1}{2n}})\). 
First, suppose \(\Khr(T)\) contains only rational components of slope $0$: In this case Theorem~\ref{thm:split_tangle_detection} implies that $T$ is horizontally split, contradicting the assumption in Theorem~\ref{thm:main}.  Next, suppose \(\Khr(T)\) only contains curves of slope 0 of which at least one is special: This case is covered by Lemma~\ref{lem:only_zero_slopes_implies_GCCC}.
The last case of \(\Khr(T)\) containing curves of non-zero slopes is covered by Lemma~\ref{lem:nonzero_slope_implies_limit_iinfinitty}.
\end{proof}

\subsection{Non-trivial band detection}

Joshua Wang asked if it is possible to recover~\cite[Theorems~1.1 and~1.8]{Wang} using our techniques:
\begin{theorem}[Split closure property]\label{thm:split_link_closure_detection}
	Suppose \(T(0)\) is a split link $K \cup K'$. Then \(\Khr(T)\) only contains components of slope \(0\).
\end{theorem}

\begin{proof}
	Placing the reduction point to the top-left end of $T$, we consider the reduced Khovanov homology of the split link \(\Khr(T(0))=\Khr(K\cup K')\). Since only one of the two components is reduced (say $K$), we may consider the basepoint action on \(\Khr(K\cup K')\) with respect to a basepoint on $K'$, which we place near the bottom left end of $T$. Keeping in mind that we work with $\F $ coefficients, we have
	\[\Khr(K\cup K')=\Khr(K) \otimes_{\F} \Kh(K') \]
	and the basepoint action  of $\F[x]/(x^2)$ on $\Khr(K\cup K')$ is induced by the basepoint action of $ \F[x]/(x^2)$ on $ \Kh(K')$. Over $\F $ the latter basepoint action on unreduced Khovanov homology is well-known to be free~\cite[Corollary~3.2.C]{Shumakovitch}, and so the basepoint action of \(\F[x]/(x^2)\) on \(\Khr(K\cup K')\) is also free.

	We now come back to the tangle $T$ and leverage the description of Khovanov homology in terms of the Floer homology of curves:
	$$\Khr(K\cup K')= \Khr(T(0))=\HF(\textcolor{red}{\a_0}, \textcolor{blue}{\Khr(T)})$$
	Let us suppose there is a curve $\textcolor{blue}{\gamma}$ in \(\textcolor{blue}{\Khr(T)}\) of slope $\nicefrac p q \neq 0$. By adding twists to the lower two punctures we may assume that $\nicefrac p q$ is positive  and as close to $0$ as we want. After pulling the curves tight in the planar cover, each intersection between $\textcolor{red}{\a_0}$ and $\textcolor{blue}{\gamma}$  locally looks as in Figure~\ref{fig:interpicture}. (Here we implicitly use the idea from Remark~\ref{rem_idea}.) 
	\begin{figure}[t]
	\centering
	\labellist 
	\pinlabel $D$ at 25 58
	\pinlabel $S$ at 57 45
	\pinlabel $S$ at 105 45
	\pinlabel $S$ at 208 105
	\pinlabel $D$ at 238 96
	\pinlabel $S$ at 167 105
	\pinlabel $D$ at 138 57
	\pinlabel \tiny$\textcolor{red}{0}$ at 154 79
	\pinlabel \tiny$\textcolor{blue}{1}$ at 9 47
	\pinlabel \tiny$\textcolor{blue}{2}$ at 40 44
	\pinlabel \tiny$\textcolor{blue}{3}$ at 122 71
	\pinlabel \tiny$\textcolor{blue}{4}$ at 154 68
	\pinlabel \tiny$\textcolor{blue}{5}$ at 222 92
	\pinlabel \tiny$\textcolor{blue}{6}$ at 254 89
	\endlabellist
	\includegraphics[scale=1.2]{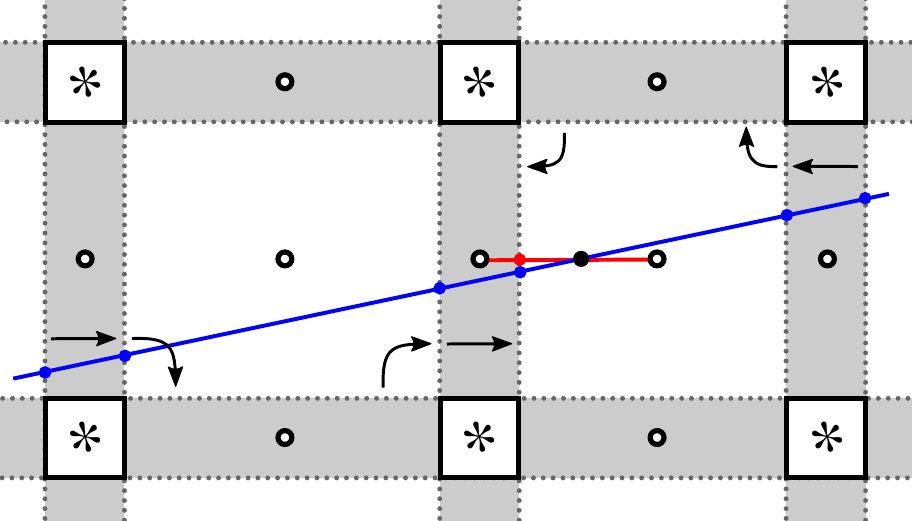}
	\caption{The lifts of the arc $\textcolor{red}{\a_0}$ and a rational or special curve $\gamma$ of sufficiently shallow positive slope, based at a common point of intersection }
	\label{fig:interpicture}
	\end{figure}
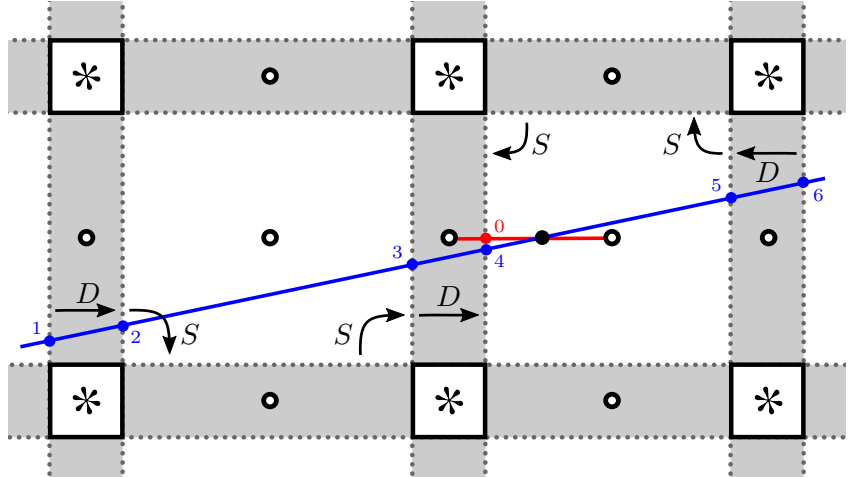

	Thus the complexes over the algebra $\mathcal B$ associated to the curves $\textcolor{red}{\a_0}$ and $\textcolor{blue}{\gamma}$ are as follows:
	\[
	\begin{aligned}
	\textcolor{red}{\DD_0} &= 
	\left[
	\begin{tikzcd}[nodes={inner sep=2pt},column sep=14pt,ampersand replacement=\&, red]
	\DotBred_0
	\end{tikzcd}
	\right]
	\hateqq \textcolor{red}{\a_0}
	\\
	\textcolor{blue}{\DD_1} &=
	\left[ 
	\begin{tikzcd}[nodes={inner sep=2pt},column sep=14pt,ampersand replacement=\&, blue]
	\cdots
	\arrow[r]
	\&
	\DotBblue_1
	\arrow{r}{D}
	\&
	\DotBblue_2
	\arrow{r}{S^2}
	\&
	\DotBblue_3
	\arrow{r}{D}
	\&
	\DotBblue_4
	\&
	\DotBblue_5
	\arrow[l,"S^2" above]
	\&
	\DotBblue_6
	\arrow[l,"D" above]
	\&
	\cdots
	\arrow[l]
	\end{tikzcd}
	\right]
	\hateqq \textcolor{blue}{\gamma}
	\end{aligned}
	\]
	where the subscripts are used to simply label generators. According to \cite[Theorem~1.5]{KWZ}, the Lagrangian Floer homology between two curves is isomorphic to the homology of the morphism space between the corresponding complexes:
	$$\HF(\textcolor{red}{\a_0},\textcolor{blue}{\gamma})\cong \Homology(\Mor(\textcolor{red}{\DD_0} ,\textcolor{blue}{\DD_1})) \subseteq \Khr(T(0))=\HF(\textcolor{red}{\a_0}, \textcolor{blue}{\Khr(T)})$$
	Consider now a morphism consisting of a single arrow $ (\textcolor{red}{\DotBred_0} \xrightarrow{\id} \textcolor{blue}{\DotBblue_4}) \in \Homology(\Mor(\textcolor{red}{\DD_0},\textcolor{blue}{\DD_1}))$. 
	(This morphism corresponds to the single intersection in Figure~\ref{fig:interpicture}.) The basepoint action multiplies all the labels of a morphism by $D_{\smallDotB}$; see the discussion before \cite[Lemma~6.46]{KWZthinness} for a detailed explanation of the basepoint action in the context of morphism spaces of chain complexes. Thus the basepoint action sends the morphism $(\textcolor{red}{\DotBred_0} \xrightarrow{\id} \textcolor{blue}{\DotBblue_4})$  to the morphism $(\textcolor{red}{\DotBred_0} \xrightarrow{D} \textcolor{blue}{\DotBblue_4})$, which is null-homotopic; the null-homotopy is $(\textcolor{red}{\DotBred_0} \xrightarrow{\id} \textcolor{blue}{\DotBblue_3})$. Furthermore, the morphism $(\textcolor{red}{\DotBred_0} \xrightarrow{\id} \textcolor{blue}{\DotBblue_4})$ is not in the image of the basepoint action, because every morphism that is homotopic to the one in the image of the basepoint action cannot contain identity arrows $\xrightarrow{\id}$. 
	This is because both $\textcolor{red}{\DD_0}$ and $\textcolor{blue}{\DD_1}$ correspond to the pulled tight curves, and thus do not contain any $\xrightarrow{\id}$ in their differentials. 
	We conclude that the morphism $(\textcolor{red}{\DotBred_0} \xrightarrow{\id} \textcolor{blue}{\DotBblue_4})$ represents torsion in the basepoint action of \(\F[x]/(x^2)\) on \( \Khr(T(0))\). So this action is not free, contradicting the fact that  \(T(0)\)  is split.
\end{proof}

\begin{theorem}
	Suppose $T$ is not horizontally split, and \(T(0)\) is a split link. Then the unoriented knots $\{K_n\}_{n\in \Z}$ are pairwise different.
\end{theorem}
\begin{proof}
	Theorem~\ref{thm:split_link_closure_detection} implies that \(\Khr(T)\) only contains components of slope 0. By Theorem~\ref{thm:split_tangle_detection} we know that the tangle~\(T\) must contain special components of slope 0. Lemma~\ref{lem:only_zero_slopes_implies_GCCC} now proves the statement.
\end{proof}

\subsection{Towards split closure detection}

It is natural to wonder if the converse of Theorem~\ref{thm:split_link_closure_detection} also holds.

\begin{conjecture}[Split closure detection]\label{conj:split_closure_detection}
	Given a Conway tangle \(T\), \(T(0)\) is a split link if and only if \(\Khr(T)\) only contains components of slope \(0\).
\end{conjecture}

In this direction, we can offer the following result:

\begin{theorem}
	\(T(0)\) is a split link if \(\Khr(T)\) only contains components of slope 0 and no rational curves of length greater than 1.
\end{theorem}

\begin{proof}
	If \(\Khr(T)\) only contains components of slope 0 then \(\Khr(T(0))\) is isomorphic to the Lagrangian Floer homology between \(\BN(\Lo)\) and the rational components of \(\Khr(T)\). If all rational components are (up to grading shift) equal to \(\r_1(0)\) this implies that the basepoint action on \(\Khr(T(0))\) is free. So by the main result of \cite{LS}, \(T(0)\) is a split link. 
\end{proof}

To prove Conjecture~\ref{conj:split_closure_detection}, it remains to show that if $\Khr(T;\F)$ only contains curves of slope 0 then all its rational components have length 1. 
Note that it is important that we use coefficients in \(\fieldTwoElements\), because this statement is false away from characteristic $2$. 
However, over $\fieldTwoElements$, we in fact expect this to be a more general property of the multicurve invariants \(\Khr\):

\begin{conjecture}
	For any Conway tangle \(T\), the length of any rational component of  \(\Khr(T;\fieldTwoElements)\) is equal to 1. 
\end{conjecture}

\newcommand*{\arxiv}[1]{\href{http://arxiv.org/abs/#1}{ArXiv:\ #1}}
\newcommand*{\arxivPreprint}[1]{\href{http://arxiv.org/abs/#1}{ArXiv preprint #1}}
\bibliographystyle{alpha}
\bibliography{main}

\end{document}